\documentclass[11pt]{amsart}

\usepackage{scitpl_ams}

\newcommand{\ds}{\displaystyle}
\newcommand{\xnorm}[2][2]{\LebNorm[#1][][\mbT]{#2}}
\newcommand{\xdual}[2]{\LebDual[2][][\mbT]{#1}{#2}}
\newcommand{\xhnorm}[2][2]{\LebNorm[#1][][\mbT]{#2}}
\newcommand{\xhdual}[2]{\LebDual[2][][\mbT]{#1}{#2}}
\newcommand{\xKnorm}[2][K]{\LebNorm[2][][#1]{#2}}
\newcommand{\xKdual}[3][K]{\LebDual[2][][#1]{#2}{#3}}
\newcommand{\xdhseminorm}[1]{\seminorm{#1}_{\mathrm{DG}}}
\newcommand{\xjhseminorm}[1]{\seminorm{#1}_{\mathrm{jmp}}}

\newcommand{\xahbimap}[2]{a_h(#1,#2)}
\newcommand{\xahadjbimap}[2]{a_h^*(#1,#2)}
\newcommand{\xbhbimap}[2]{b_h(#1,#2)}
\newcommand{\xbhadjbimap}[2]{b_h^*(#1,#2)}

\begin{document}


\title[Discrete hypocoercive estimates for discontinuous Galerkin
  methods]{Discrete hypocoercive estimates for discontinuous Galerkin
  methods:  application to the Vlasov-Poisson-Fokker-Planck system}

\author{Yi Cai}
\address{School of Mathematical Sciences,
University of Science and Technology of China,
Hefei, Anhui, 230026, PR China}
\email{caiyi0210@ustc.edu.cn}

\author{Alain Blaustein}
\address{Alain Blaustein, Centre Inria de l'Université de Lille,
Villeneuve-d'Ascq, F-59650, France}
\email{alain.blaustein@inria.fr}

\author{Tao Xiong}
\address{Tao Xiong, School of Mathematical Sciences \& State Key Laboratory of Cognitive Intelligence,
University of Science and Technology of China,
Hefei, Anhui, 230026, PR China}
\email{taoxiong@ustc.edu.cn}

\author{Francis Filbet}
\address{Francis Filbet, Institut de Mathématiques de Toulouse,
Université de Toulouse, 118, route de Narbonne, 
Toulouse F-31062, France}
\email{francis.filbet@math.univ-toulouse.fr}

\thanks{The fourth author is the corresponding author.}


\date{}

\dedicatory{}


\begin{abstract}
We develop and analyze a class of structure-preserving discontinuous Galerkin schemes for the nonlinear Vlasov-Poisson-Fokker-Planck model, reformulated as a hyperbolic system through a Hermite expansion in the velocity variable. We discretize the Vlasov-Fokker-Planck equation with the discontinuous Galerkin method, while the Poisson equation is approximated with either a discontinuous Galerkin method or a Raviart-Thomas mixed finite element method. We prove the exponential relaxation to equilibrium for small initial data, uniformly with respect to the discretization parameters thanks to discrete hypocoercivity arguments. Moreover, we check that the resulting semi-discrete schemes preserve the conservation of mass and, depending on the discretization of the Poisson equation, the momentum or the energy balance, along with the $L^2$ variational structure of the linearized model. Numerical simulations verify the accuracy and the long-time behavior of the scheme.
\end{abstract}

\maketitle

\tableofcontents


\section{Introduction}
\label{sec:introduction}
The Vlasov-Poisson-Fokker-Planck model \cite{bouchut_long_1995} provides a kinetic description of the evolution of electrons interacting with a background of heavy, immobile positive ions through a self-consistent electrostatic field. In a $d$-dimensional periodic box $\mbT^d$, the model reads
\begin{equation} \label{eq:vpfp}
    \left\{
    \begin{array}{l}
        \ds \partial_t f \,+\, \mfv \cdot \nabla_\mfx f \,+\, \frac{q}{m} \,\mfE \cdot \nabla_\mfv f \,=\, \frac{1}{\tau_0} \,\nabla_\mfv \cdot (\mfv f \,+\, T_0\,\nabla_\mfv f), 
        \\ [0.9em]
        \ds \mfE \,=\, -\nabla_\mfx \Phi \qc -\varepsilon_0\,
      \Delta_\mfx \Phi \,=\, q\; n \,-\, q\, n_i \qc n \,=\, \int_{\mbR^d} f \,\dd{\mfv}\,, 
    \end{array}
    \right.
\end{equation}
where $f(t,\mfx,\mfv)$ denotes the electron distribution function in
phase space $(\mfx,\mfv)\in\mbT^d\times\mbR^d$ at time $t\ge 0$,
whereas $\mfE(t,\mfx)$ is generated by the electrostatic potential
$\Phi(t,\mfx)$ solution to the Poisson equation. This coupling
involves several physical parameters, including the vacuum
permittivity $\varepsilon_0$, the electron charge
  $q<0$, and the electron mass $m$. The electron  density $n(t,\mfx)$ is defined as the velocity integral of $f$, while $n_i(\mfx)$ is the prescribed ion density. Thermodynamic effects are modeled by the Fokker-Planck operator on the right-hand side of the kinetic equation, which accounts for collisions between electrons and a fixed ionic background at a spatially homogeneous temperature $T_0>0$. The parameter $\tau_0>0$ denotes the associated mean free time. The Poisson equation is supplemented with the quasi-neutrality condition 
\begin{equation*}
    q \int_{\mbT^d} n(t,\mfx)\;\dd{\mfx} \,-\, q \,\int_{\mbT^d} n_i(\mfx)\,\dd{\mfx} \,=\,0 \qc \forall \,t \,\ge\, 0,
\end{equation*}
and the uniqueness condition for the electric potential
\begin{equation} \label{eq:vpfp_uniqueness}
    \int_{\mbT^d} \Phi(t,\mfx)\,\dd{\mfx} \,=\, 0 \qc \forall\, t \,\ge\, 0.
\end{equation}
In the long-time regime $t\rightarrow +\infty$, the distribution function $f(t,\mfx,\mfv)$ converges toward a stationary state $f_\infty(\mfx,\mfv)$, defined by
\begin{equation*}
    f_\infty(\mfx,\mfv) = n_\infty(\mfx)\mathcal{M}(\mfv),
\end{equation*}
where $\mathcal{M}(\mfv)$ is the Maxwellian with temperature $T_0$,
\begin{equation*} 
    \mathcal{M}(\mfv) \,=\, \ds\frac{1}{(2\pi T_0)^{\frac{d}{2}}}\,\exp\left(-\frac{\abs{\mfv}^2}{2\,T_0}\right),
\end{equation*}
and where the macroscopic density $n_\infty(\mfx)$ is determined by the Poisson-Boltzmann equation
\begin{equation*} 
    n_\infty =
    c_\infty\exp\left(-\frac{q}{m T_0}\Phi_\infty\right) \qc
    -\varepsilon_0\Delta_{\mfx}\Phi_\infty =
q\,n_\infty - q\,n_i ,
\end{equation*}
together with the uniqueness condition \eqref{eq:vpfp_uniqueness} written for $\Phi_\infty(\mfx)$.
The normalization constant $c_\infty$ is uniquely fixed by the quasi-neutrality condition
\begin{equation*}
    c_\infty \int_{\mbT^d}\exp\left(-\frac{q}{m T_0}\Phi_\infty(\mfx)\right) \dd{\mfx} = \int_{\mbT^d}n_i(\mfx)\dd{\mfx}. 
\end{equation*} 
The above asymptotic behavior means that the electrons' temperature relaxes to the background temperature and their distribution converges to a Maxwell-Boltzmann distribution. 

Numerical simulation of the system \eqref{eq:vpfp} faces the typical challenges of kinetic equations, namely multiple scales and the high dimensionality of the phase space. A variety of numerical methods have been developed for this system; see, for instance, \cite{havlak_deterministic_1998,wollman_numerical_2005,ayuso_discontinuous_2011, crestetto_kinetic_2012,degond_asymptoticpreserving_2017,manzini_convergence_2017, carrillo_variational_2021,ye_energyconserving_2024}. These approaches aim to capture physical phenomena arising in weakly collisional plasmas, such as Landau damping and the two-stream instability, which typically occur over short time intervals before being damped by collisional effects. To address high dimensionality, dynamical low-rank algorithms have been proposed \cite{camporeale_velocity_2016, einkemmer_mass_2021,coughlin_efficient_2022,guo_low_2022} to reduce computational cost via dimension splitting and singular value decomposition.

Recently, Blaustein and Filbet proposed a finite volume scheme based on Hermite polynomials in the velocity variable for the nonlinear Vlasov-Poisson-Fokker-Planck model \cite{blaustein_discrete_2024, blaustein_structure_2024,blaustein_structure_2025}. Using hypocoercivity arguments, they established exponential relaxation to equilibrium in a weighted $ L^2$-functional framework for the linearized model. In the collisionless setting, discontinuous Galerkin methods combined with Hermite decompositions have been successfully applied to the Vlasov-Poisson system \cite{kormann_generalized_2021, filbet_conservative_2022,bessemoulin-chatard_stability_2022,bessemoulin-chatard_convergence_2023}. These methods naturally conserve mass, and can be extended to conserve momentum and energy. In addition, discontinuous Galerkin discretizations offer several practical advantages, including high-order accuracy, compactness, $h$-$p$ adaptivity, parallel efficiency, and flexibility in handling complex geometries \cite{biswas_parallel_1994,shu_discontinuous_2009,cangiani_hpversion_2017}. 

The purpose of our investigation is to develop a discontinuous Galerkin framework that captures the correct long-time behavior of the nonlinear Vlasov-Poisson-Fokker-Planck model in both weakly and strongly collisional regimes. From now on, we focus on the one-dimensional model with a homogeneous ionic background $\rho_0\in \mathbb{R}_+^*$. For simplicity, we rescale the variables so that the physical parameters $q$, $m$, and $\varepsilon_0$ are absorbed into the scaling, and we change $(\mfE,\Phi)$ into $(-\mfE,-\Phi)$. Under this normalization, the system \eqref{eq:vpfp} reduces to
\begin{equation*} 
    \left\{
    \begin{array}{l}
       \ds \partial_t f \,+\, v\,\partial_x f \,+\,  E\,\partial_v f \,=\, \frac{1}{\tau_0}\,\partial_v\left(v\,f \,+\, T_0\,\partial_vf\right),
       \\[0.9em]
       \ds E \,=\, -\partial_x\Phi \qc -\partial_{xx}^2\Phi \,=\, \rho \,-\, \rho_{0} \qc \rho \,=\, \int_{\mbR} f\,\dd{v},
    \end{array}
    \right.
\end{equation*}
with $(x,v)\in\mbT\times\mbR$ and $t\ge 0$, supplemented with the quasi-neutrality condition 
\begin{equation}
  \label{eq:vpfp_quasineutrality}
    \int_{\mbT} \rho(t,x)\;\dd{x} \,=\, \left|\mbT\right|\, \rho_0 \,\qc \forall \,t \,\ge\, 0,
  \end{equation}
  and the equilibrium state $f_\infty$ reads
\begin{equation}\label{def:M}
	f_\infty(x,v) = \rho_\infty\mathcal{M}(v),\quad \textrm{with}\quad
	\mathcal{M}(v) \,=\, \ds\frac{1}{\sqrt{2\pi T_0}}\,e^{-\frac{\abs{v}^2}{2T_0}}\quad \textrm{and}\quad \rho_\infty\,=\,\rho_0.
\end{equation}
Since we are interested in the large time behavior of the solutions near $f_\infty$, we decompose the nonlinear field interaction $E\partial_vf$ into a linearized component $E\partial_vf_\infty$ and a nonlinear remainder $E\partial_v(f-f_\infty)$, leading to
\begin{equation} \label{eq:vpfp1d}
	\left\{
	\begin{array}{l}
		\ds \partial_t f \,+\, v\,\partial_x f \,+\, E\,\partial_vf_\infty \,+\,  E\,\partial_v(f-f_\infty) \,=\, \frac{1}{\tau_0}\,\partial_v\left(v\,f \,+\, T_0\,\partial_vf\right),
		\\[0.9em]
		\ds E \,=\, -\partial_x\Phi \qc -\partial_{xx}^2\Phi \,=\, \rho \,-\, \rho_{\infty} \qc \rho \,=\, \int_{\mbR} f\,\dd{v}.
	\end{array}
	\right.
\end{equation}
Let us observe that after removing the nonlinear component, the resulting system corresponds to the linearized Vlasov-Poisson-Fokker-Planck model. It has been analyzed at the discrete level in \cite{blaustein_structure_2024} within a finite volume framework for the spatial variable and a spectral Hermite discretization for the velocity variable. The main analytical tools to prove the large time convergence of the discrete approximations at the linearized level are the so-called discrete hypocoercivity methods, developed in \cite{porretta_numerical_2017,bessemoulin-chatard_hypocoercivity_2020, dujardin_coercivity_2020, georgoulis_hypocoercivitycompatible_2021, blaustein_discrete_2024, blaustein_structure_2024, blaustein_structure_2025}, which themselves build on the continuous theory of hypocoercivity developed in \cite{villani_hypocoercivity_2009,dolbeault_hypocoercivity_2015}. In this article, we extend this strategy to the discontinuous Galerkin framework and to the approximations of the fully nonlinear system \eqref{eq:vpfp1d}.
The main additional difficulty is that the linear variational structure on which hypocoercivity methods rely no longer holds at the nonlinear level. Therefore, to extend discrete hypocoercivity methods to the nonlinear level, we are led to control the additional nonlinear terms in the discontinuous Galerkin framework.

\vspace{1em}

The remainder of the paper is organized as follows. In Section~\ref{sec:disc_vfp}, we reformulate the nonlinear Vlasov-Fokker-Planck equation using a Hermite decomposition for the velocity variable and introduce a spatial discretization  using a local discontinuous Galerkin method. Then, we focus on the approximation of the Poisson equation in Section \ref{sec:disc_poisson}, for which we propose both discontinuous Galerkin and mixed finite element methods. In Section \ref{sec:main_results}, we present the main results on the combined scheme together with their proofs. Finally, numerical simulations are presented in Section~\ref{sec:simulation}, and concluding remarks follow in Section~\ref{sec:conclusion}.


\section{Discretization of Vlasov-Fokker-Planck equation}
\label{sec:disc_vfp}

In this section, we focus on the discretization of the kinetic Vlasov-Fokker-Planck equation.  We first apply a spectral method for the velocity variable based on a Hermite expansion. Then, we apply a local discontinuous Galerkin method based on alternating fluxes. This approach preserves the structure of the continuous equation, allowing us to derive discrete energy and hypocoercive estimates.


\subsection{Hermite decomposition for the velocity variable}
\label{sec:velocity}
The purpose of this section is to present a reformulation of the Vlasov-Poisson-Fokker-Planck model based on a spectral decomposition as in \cite{blaustein_discrete_2024, blaustein_structure_2024,blaustein_structure_2025}. More precisely, we expand the distribution function $f$ as
\begin{equation} \label{eq:hermite_decomposition}
    f(t,x,v) \,=\, \sum_{k\in\mbN} D_k(t,x)\,e_k(v),
\end{equation}
where $(e_k)_{k\in\mbN}$ is the orthonormal basis of $L^2(\mathcal{M}^{-1}\dd{v})$ defined by 
\begin{equation*} 
    e_k(v) \,=\, H_k\left(\frac{v}{\sqrt{T_0}}\right)\,\mathcal{M}(v),
\end{equation*}
and where $\mathcal{M}$ is given in \eqref{def:M} and $(H_k)_{k\in\mbN}$ denotes the family of normalized Hermite polynomials orthonormal with respect to the standard Gaussian weight $e^{-\xi^2/2}/\sqrt{2\pi}$. They satisfy the recurrence relation 
\begin{equation*} 
	\xi \,H_k(\xi) \,=\, \sqrt{k}\,H_{k-1}(\xi) \,+\, \sqrt{k+1}\,H_{k+1}(\xi) \qc \forall\; k \,\in\, \mbN\,,
\end{equation*}
with $H_{-1} = 0$ and $H_{0} = 1$. It is also worth noting that $(e_k)_{k\in\mbN}$ diagonalizes the Fokker-Planck operator:
\begin{equation*}
    \partial_v(v\,e_k \,+\, T_0\,\partial_v e_k) \,=\, -k\,e_k \qc \forall \,k\,\in\,\mbN,
\end{equation*}
which follows from the recurrence relation and the differential relation of Hermite polynomials,
\begin{equation*} 
    H_k'(\xi) \,=\, \sqrt{k}\,H_{k-1}(\xi) \qc \forall k\in\mbN.
\end{equation*}
We check that the coefficients $D_\infty= (D_{\infty,k})_{k\in\mbN}$ in the expansion \eqref{eq:hermite_decomposition} of the equilibrium $f_\infty$ are given by
\begin{equation} \label{eq:vpfp1d_hermite_equilibrium}
    D_{\infty,k} = \begin{cases}
        \ds \rho_\infty, & \text{if } k = 0; \\
        \ds 0,  &\text{else.}
    \end{cases}
\end{equation}
Inserting \eqref{eq:hermite_decomposition} into \eqref{eq:vpfp1d}, we obtain the following system for $(D_k)_{k\in\mbN}$,
\begin{equation*} 
    \left\{
    \begin{array}{l}
         \ds \partial_t D_0 \,-\, \mathcal{A}^* \,D_1 \,=\, 0, 
         \\ [0.9em]
         \ds \partial_t D_1 \,+\, \mathcal{A} \,D_0 \,-\, \sqrt{2}\, \mathcal{A}^*D_2 \,+\, \frac{1}{T_0} \,\mathcal{A}\,\Phi\,D_{\infty,0}  \,-\, \sqrt{\frac{1}{T_0}}\,E\,(D_0 - D_{\infty,0}) \,=\, -\frac{1}{\tau_0}\,D_1, 
         \\ [0.9em]
         \ds \partial_t D_k \,+\, \sqrt{k}\,\mathcal{A}\, D_{k-1} \,-\, \sqrt{k+1}\,\mathcal{A}^* D_{k+1} \,-\, \sqrt{\frac{k}{T_0}}\,E\,D_{k-1} \,=\, -\frac{k}{\tau_0}\,D_k \qc \forall \,k \,\ge\, 2\,,
    \end{array}
    \right.
\end{equation*}
where $\mathcal{A} = \sqrt{T_0}\,\partial_x$ denotes the spatial differential operator and $\mathcal{A}^* = -\sqrt{T_0}\,\partial_x$ is its adjoint with respect to the $L^2(\mbT)$ inner product. The linearized term $E\,\partial_vf_\infty$ corresponds to $\frac{1}{T_0}\,\mathcal{A}\,\Phi \,D_{\infty,0}$ in the equation on $D_1$, while the nonlinear term $E\,\partial_v(f-f_\infty)$ corresponds to the last term on the left hand side of the second and third lines in the previous system. To close this system, we rewrite the Poisson equation, using that $H_0 = 1$, as
$$
E \,=\, -\partial_{x}\Phi \qc \,-  \partial_{xx}^2 \Phi \,=\, D_0 \,-\, D_{\infty,0}\,,
$$
while the quasi-neutrality condition \eqref{eq:vpfp_quasineutrality} becomes
\begin{equation*} 
    \xdual{D_0 - D_{\infty,0}}{1} = 0\,,
\end{equation*}
and the uniqueness condition \eqref{eq:vpfp_uniqueness} is
\begin{equation*}
    \xdual{\Phi}{1} = 0\,.
\end{equation*}
We now discretize the velocity variable by retaining only the first $N_H+1$ Hermite modes $(D_{k})_{0\leq k\leq N_H }$, which leads to the following truncated system
\begin{equation} \label{eq:vpfp1d_hermite:disc}
	\left\{
	\begin{array}{l}
		\ds \partial_t D_0 - \mathcal{A}^* D_1 = 0, 
		\\ [0.5em]
		\ds \partial_t D_1 + \mathcal{A} D_0 - \sqrt{2}\, \mathcal{A}^*D_2 + \frac{1}{T_0} \mathcal{A}\Phi\, D_{\infty,0} - \sqrt{\frac{1}{T_0}}\,E\,(D_0 - D_{\infty,0}) = -\frac{1}{\tau_0}D_1, 
		\\ [0.8em]
		\ds \partial_t D_k + \sqrt{k}\,\mathcal{A} D_{k-1} - \sqrt{k+1}\,\mathcal{A}^* D_{k+1} - \sqrt{\frac{k}{T_0}}\,E\,D_{k-1} = -\frac{k}{\tau_0}D_k\,,\quad k \in \{2\,,\dots\,,N_H \},
		\\ [0.8em]
		\ds E = -\partial_{x}\Phi \qc -  \partial_{xx}^2 \Phi = D_0 - D_{\infty,0}\,,
	\end{array}
	\right.
\end{equation}
which is closed by taking $D_{N_H+1}=0$, where we assume $N_H\geq 2$. 


\subsection{Discontinuous Galerkin discretization for the spatial variable} 
\label{sec:spatial}

In this subsection, we turn to the spatial discretization of the system \eqref{eq:vpfp1d_hermite:disc}. We begin by dividing the domain $\mbT$ into a finite collection $\mcT_h $ of elements $K_j = [x_{j-\half},x_{j+\half}]$ with $j \in \mcJ = \set{1,\dots,N_x}$, where 
\begin{equation*}
    x_{\half} < x_{\frac32} < \cdots < x_{N_x-\frac12} < x_{N_{x}+\frac12},
\end{equation*}
and, by periodicity, $x_{\half}$ and $x_{N_x+\half}$ are identified, as well as the indices $j=0$ and $j=N_x$.
We also denote by $h_j = x_{j+\half} - x_{j-\half}$ the cell size and $h = \max\limits_{j\in\mcJ}h_j$ the mesh size. The mesh $\mcT_h$ is assumed to be quasi-uniform, namely, there exists a constant $C_{qu} > 0$ independent of $h$ such that 
\begin{equation*} 
    h \le C_{qu}\min_{j\in\mcJ} h_j\,.
\end{equation*}
On this mesh, we define the piecewise polynomial space $U_h^m \subset L^2(\mbT)$ for any $m \in \mbN$ as follows:
\begin{equation*}
     U_h^m = \set{u\in L^2(\mbT)\setmid \restr{u}{K_j}\in \mathscr{P}_m(K_j)\qc \forall j\in\mcJ}, 
\end{equation*}
where the local space $\mathscr{P}_m(K_j)$ consists of polynomials of degree up to $m$ on $K_j$. The functions in $U_h^m$ may be discontinuous across cell interfaces. Therefore, we denote by $u^-$ and $u^+$ the left and right limits of the function $u$ at $x$, respectively,
\begin{equation*}     
u^\pm(x) \,=\, \lim_{\delta\goto 0^+}u(x\pm\delta). 
\end{equation*}
Then, the average $\ave{\cdot}$ and jump $\jmp{\cdot}$ of the function $u$ at $x$ are given by
\begin{equation*}
    \ave{u}(x) = \half\bigl(u^+(x) + u^-(x)\bigr)\qc \jmp{u}(x) = u^+(x) - u^-(x).
\end{equation*}

We now focus on the discretization of the first three equations in
\eqref{eq:vpfp1d_hermite:disc}, whereas the approximation
$(E_h,\Phi_h)$ of the electric field and potential, with $\Phi_h\in U_h^m$, will be specified in the next section. For given initial data $D_h^{\textrm{in}}\,=\,\left(D_{h,k}^{\textrm{in}}\right)_{0\leq k\leq N_H}$ such that for any $k \in \{0,\dots , N_H\}$,
\[
D_{h,k}^{\textrm{in}}\in U^m_h\,,
\]
we seek $D_{h,k}(t) \in U_h^m$ such that for all test functions $u_k \in U_h^m$,
\begin{equation} \label{eq:vpfp1d_hermite_semidiscrete_vfp}
    \left\{
    \begin{array}{l}
         \ds \xhdual{\partial_t D_{h,0}}{u_0} - \xahadjbimap{D_{h,1}}{u_0} = 0,  
         \\ [0.8em]
         \ds \begin{aligned}
             \xhdual{\partial_t D_{h,1}}{u_1} 
             &+\; \xahbimap{D_{h,0}}{u_1} \,-\, \sqrt{2}\,\xahadjbimap{D_{h,2}}{u_1} \;+\; \frac{D_{\infty,0}}{T_0}\xahbimap{\Phi_h}{u_1} 
             \\[0.8em]
             &- \sqrt{\frac{1}{T_0}}\;\xhdual{E_h(D_{h,0}\;-\;D_{\infty,0})}{u_1} \;=\; -\frac{1}{\tau_0}\;\xhdual{D_{h,1}}{u_1},
         \end{aligned}
         \\ [3.5em]
         \ds \begin{aligned}
             \xhdual{\partial_t D_{h,k}}{u_k} 
             &+\; \sqrt{k}\,\xahbimap{D_{h,k-1}}{u_k} \;-\; \sqrt{k+1} \,\xahadjbimap{D_{h,k+1}}{u_k} 
             \\[0.8em]
             &-\; \sqrt{\frac{k}{T_0}}\;\xhdual{E_h D_{h,k-1}}{u_k} \;=\; -\frac{k}{\tau_0}\;\xhdual{D_{h,k}}{u_k}\,,
         \end{aligned}
    \end{array}
    \right.
\end{equation}
for $k\in\{2,\ldots,N_H\}$ and $D_{h,N_H+1}=0$, with the initial condition $D_{h,k}(0)=D^{\textrm{in}}_{h,k}$ for $k\in\{0,\dots,N_H\}$.
The bilinear forms $\xahbimap{\cdot}{\cdot}$ and $\xahadjbimap{\cdot}{\cdot}$ respectively provide a consistent approximation of the operators $\mathcal{A}$ and $\mathcal{A}^*$ in the weak sense. More precisely, they are defined for all $(D,u) \in U_h^m\times U_h^m$ as follows:
\begin{equation} \label{eq:ops_ah_global}
    \left\{
    \begin{array}{l}
        \ds \xahbimap{D}{u} \;\defeq\; -\sqrt{T_0}\left(\sum_{j\in\mcJ}\hat{g}_{a,j-\half}(D)\jmp{u}_{j-\half} + \sum_{j\in\mcJ}\xKdual[K_j]{D}{\partial_x u}\right),
        \\ [1.5em]
        \ds \xahadjbimap{D}{u} \;\defeq\;  +\sqrt{T_0}\left(\sum_{j\in\mcJ}\hat{g}_{a,j-\half}^*(D)\jmp{u}_{j-\half} + \sum_{j\in\mcJ}\xKdual[K_j]{D}{\partial_xu}\right).
    \end{array}
    \right.
\end{equation}
The numerical fluxes $(\hat{g}_a,\hat{g}_a^*)$ are chosen to be alternating fluxes
\begin{equation}
  \label{eq:ops_ah_flux}
   \hat{g}_a(D) = D^- \qc \hat{g}_a^*(D) = D^+ \quad\text{or}\quad \hat{g}_a(D) = D^+ \qc \hat{g}_a^*(D) = D^-.
 \end{equation}

It is worth mentioning that inserting \eqref{eq:ops_ah_flux} into \eqref{eq:ops_ah_global} and summing over $j \in \mcJ$, we get that
\begin{equation*}
    \begin{aligned}
        \xahbimap{u}{v} - \xahadjbimap{v}{u} =
        -\sqrt{T_0}\sum_{j\in\mcJ} \left(\hat{g}_{a,j-\half}(u) \jmp{v}_{j-\half} + \jmp{u}_{j-\half}\hat{g}^*_{a,j-\half}(v) - \jmp{uv}_{j-\half}\right) = 0,
    \end{aligned}
\end{equation*}
where we used
$$
\sum_{j}\int_{K_j}\partial_x(uv)\,\dd{x} = -\sum_j\jmp{uv}_{j-\half}
$$
by periodicity and the identity
$u^-\jmp{v}+v^+\jmp{u}=u^+\jmp{v}+v^-\jmp{u}=\jmp{uv}$. It then yields the preservation of the duality property
\begin{equation}
\label{ah:dual}
\xahbimap{u}{v} \,=\, \xahadjbimap{v}{u}, \quad \forall \,u,\,v \,\in U_h^m. 
\end{equation}
Moreover, we set $D = D_{\infty,0}$ in \eqref{eq:ops_ah_global} with \eqref{eq:ops_ah_flux} to get
\begin{equation*}
    \xahbimap{D_{\infty,0}}{u} \,=\, -\sqrt{T_0}\,D_{\infty,0}\left(\sum_{j\in\mcJ}\jmp{u}_{j-\half} + \sum_{j\in\mcJ}\int_{K_j}\partial_x u\,\dd{x}\right),
\end{equation*}
which yields, since
$$
\sum_{j}\int_{K_j}\partial_x u\,\dd{x} = -\sum_j\jmp{u}_{j-\half}
$$
by periodicity,
\begin{equation} \label{ah:kernel}
    \xahbimap{D_{\infty,0}}{u} = 0 \qc \forall u\in U_h^m.
\end{equation}
Finally, we emphasize that the choice of discretization of the linearized term $E\,\partial_vf_\infty$ as $D_{\infty,0}\,\xahbimap{\Phi_h}{u_1}/T_0$ in the second line of \eqref{eq:vpfp1d_hermite_semidiscrete_vfp} will play a fundamental role in preserving the linearized energy structure (see Proposition \ref{prop:4:2}).


\section{Discretization of Poisson equation}
\label{sec:disc_poisson}

We move to the discretization of the Poisson equation, where we aim for flexibility by treating various cases. Given an approximation $D_{h,0}$ of $D_0$, we look for $E_h \in W_h$ and $\Phi_h \in V_h$ such that for all test functions $w \in W_h$ and $v \in V_h$,
\begin{equation} \label{eq:vpfp1d_hermite_semidiscrete_poisson}
\left\{
\begin{array}{ll}
    \ds\xhdual{E_h}{w} &= -\xbhbimap{\Phi_h}{w}\,, \\[0.9em] 
    \ds-\xbhadjbimap{E_h}{v} &= \xhdual{D_{h,0}-D_{\infty,0}}{v},
    \end{array}\right.
\end{equation}
where $V_h$ and $W_h$ are finite-dimensional subspaces of $L^2(\mbT)$ specified later, with $V_h\subseteq U^m_h$ for the second line in \eqref{eq:vpfp1d_hermite_semidiscrete_vfp} to be well defined. The bilinear forms $\xbhbimap{\cdot}{\cdot}$ and $\xbhadjbimap{\cdot}{\cdot}$ again approximate the operators $\partial_x$ and $-\partial_x$ respectively in the weak sense. This system is closed by the compatibility condition (quasi-neutrality) 
\begin{equation*}     
\xhdual{D_{h,0} - D_{\infty,0}}{1} = 0\,,
\end{equation*}
and the uniqueness condition
\begin{equation*} 
    \xhdual{\Phi_h}{1} = 0\,.
\end{equation*}

Here we propose two different discretizations for the Poisson equation:
local discontinuous Galerkin  and Raviart-Thomas methods. As we shall show later in Section \ref{sec:preliminary_results}, these discretizations preserve the elementary properties of the continuous equation (coercivity, $L^\infty$ estimate) so that the large time behavior of the solution may be studied for the discrete nonlinear system.  

\subsection{Local discontinuous Galerkin method} 
In the local discontinuous Galerkin case, proposed in \cite{cockburn_local_1998}, the functional spaces $V_h$ and $W_h$ are 
\[V_h = \{v\in U_h^m\,|\, \xhdual{v}{1}\,=\,0\}\,,\quad \textrm{and}\quad W_h = U_h^m\,.\]
The bilinear forms are defined for all $(\Phi,w) \in V_h\times W_h$ and $(E,v) \in W_h\times V_h$ by 
\begin{equation} \label{eq:ops_bh_ldg_global}
    \left\{
    \begin{array}{l}
        \ds \xbhbimap{\Phi}{w} \;\defeq\, -\left(\sum_{j\in\mcJ}\hat{g}_{b,j-\half}(\Phi)\jmp{w}_{j-\half} + \sum_{j\in\mcJ}\xKdual[K_j]{\Phi}{\partial_x w}\right),
        \\ [1.5em]
        \ds \xbhadjbimap{E}{v} \,\defeq\, +\left(\sum_{j\in\mcJ}\hat{g}_{b,j-\half}^*(E)\jmp{v}_{j-\half} + \sum_{j\in\mcJ}\xKdual[K_j]{E}{\partial_x v}\right).
    \end{array}
    \right.
\end{equation}
Here, we choose the alternating fluxes:
\begin{equation} \label{eq:ops_bh_ldg_flux}
    \hat{g}_b(\Phi)  = \Phi^- \qc \hat{g}_b^*(E)  = E^+ \quad\text{or}\quad \hat{g}_b(\Phi)  = \Phi^+ \qc \hat{g}_b^*(E)  = E^-.
\end{equation}

\subsection{Raviart-Thomas method}
Another approach is the mixed finite element method based on the one-dimensional Raviart-Thomas element, as discussed in \cite{bessemoulin-chatard_stability_2022}. In this case, the functional spaces $V_h$ and $W_h$ are 
\[V_h = \{v\in U_h^m\,|\, \xhdual{v}{1}\,=\,0\}\,,\quad \textrm{and}\quad W_h = U_h^{m+1} \cap C^0(\mbT)\,.\]
The bilinear forms $(b_h,b_h^*)$ are defined for all $(\Phi,w) \in V_h\times W_h$ and $(E,v) \in W_h\times V_h$ by
\begin{equation} \label{eq:ops_bh_rt}
\left\{
\begin{array}{l}
    \ds\xbhbimap{\Phi}{w} \,\defeq\; -\sum_{j\in\mcJ}\xKdual[K_j]{\Phi}{\partial_x w}\,, \\[0.9em] 
    \ds\xbhadjbimap{E}{v} \;\defeq\, -\sum_{j\in\mcJ}\xKdual[K_j]{\partial_x E}{v}.\end{array}\right.
\end{equation}
It is easy to verify that, in both the Raviart-Thomas and the discontinuous Galerkin framework, $b_h$ and $b_h^*$ are constructed such that the duality property \eqref{ah:dual} is satisfied{, that is
	\begin{equation*}
		\label{bh:dual}
		\xbhbimap{v}{w} \,=\, \xbhadjbimap{w}{v}, \quad \forall \,(v,\,w) \,\in V_h\times W_h\,.
	\end{equation*}
	}
In both the local discontinuous Galerkin and the Raviart-Thomas case, the system
  \eqref{eq:vpfp1d_hermite_semidiscrete_poisson} is nonsingular on
  $V_h\times W_h$: it is square, and for a vanishing right-hand
  side only $(E_h,\Phi_h)=(0,0)$ solves it. \\
  We detail the local discontinuous Galerkin case for the choice $\hat
  g_b(\Phi)=\Phi^-$ of \eqref{eq:ops_bh_ldg_flux}, the other choice
  giving the same conclusions with opposite signs. Since $1\in W_h$
  and $\xbhbimap{\Phi_h}{1}=0$, the first line gives
  $\xhdual{E_h}{1}=0$, hence $E_h\in V_h$; testing the second line
  with $v=E_h$ and using the identity \eqref{identity:Eh:minus} established in Appendix A  yields that $\jmp{E_h}_{j-\half}=0$ for
  all $j\in\mcJ$, so that $E_h$ is continuous, the second line then
  reads $\xhdual{\partial_xE_h}{v}=0$ for all $v\in V_h$ and forces
  $\partial_xE_h=0$, and the first line tested with $w=E_h$ gives
  $E_h=0$.  Repeating the same computations for $\Phi_h$ and
  using $\xhdual{\Phi_h}{1}=0$ give that $\Phi_h=0$.
  \\
  The Raviart-Thomas system \eqref{eq:ops_bh_rt} is the classical mixed method for the one-dimensional Poisson equation, for which the same conclusion holds. Consequently \eqref{eq:vpfp1d_hermite_semidiscrete_poisson} determines $(E_h,\Phi_h)\in W_h\times V_h$ uniquely for every datum $D_{h,0}$.

We will see that these approximations also guarantee the desired discrete $L^\infty$ estimate of $E_h$. Simultaneously, the  coercivity is also preserved at the discrete level.


\section{Trend to equilibrium and invariants for the discrete system}
\label{sec:main_results}


\subsection{Strong reformulation of the discrete system}
For given initial data $D_h^{\textrm{in}}\,=\,\left(D_{h,k}^{\textrm{in}}\right)_{0\leq k\leq N_H}$ such that
\[
D_{h,k}^{\textrm{in}}\in U^m_h\,,\quad \forall \,k \in \{0,\dots , N_H\}\,,
\]
we rewrite the combined scheme \eqref{eq:vpfp1d_hermite_semidiscrete_vfp} and \eqref{eq:vpfp1d_hermite_semidiscrete_poisson} using the Riesz representation
\begin{equation} 
\label{eq:vpfp1d_hermite_semidiscrete_riesz}
	\left\{
	\begin{array}{l}
		\ds \partial_t D_{h,0} \,-\, \mathcal{A}_h^*\,D_{h,1} \,=\, 0,
		\\ [1.em]
		\ds \partial_t D_{h,1} + \mathcal{A}_h \,D_{h,0} - \sqrt{2}\,\mathcal{A}_h^*\,D_{h,2} + \frac{D_{\infty,0}}{T_0}\mathcal{A}_h\Phi_h - \sqrt{\frac{1}{T_0}}\Pi_h\left( E_h(D_{h,0}-D_{\infty,0})\right) = -\frac{1}{\tau_0}D_{h,1},
		\\ [1.em]
		\ds \partial_t D_{h,k} +\sqrt{k}\,\mathcal{A}_hD_{h,k-1} - \sqrt{k+1} \,\mathcal{A}_h^* D_{h,k+1} - \sqrt{\frac{k}{T_0}} \Pi_h\left( E_hD_{h,k-1}\right) = -\frac{k}{\tau_0} D_{h,k} \qc 2\leq k \leq N_H\,,
		\\ [1.2em]
		\ds E_h \,=\, -\mathcal{B}_h\,\Phi_h \qc -\mathcal{B}_h^*\,E_h \,=\, D_{h,0}\,-\;D_{\infty,0}\qc     \xhdual{\Phi_h}{1} \,=\, 0\,,
		\\ [1.em]
		\ds D_h(t=0)\,=\, D_{h}^{\textrm{in}}\,,
	\end{array}
	\right.
\end{equation}
with $D_{h,N_H+1}=0$ and where $\Pi_h$ denotes the $L^2$ orthogonal projection operator on the space $U_h^m$ and  the discrete operators $\mathcal{A}_h$, $\mathcal{A}_h^*: U_h^m \goto U_h^m$ are defined for all $(D,u) \in U_h^m\times U_h^m$ by
\begin{equation} 
\label{eq:ops_ah_riesz}
\left\{
\begin{array}{l}
    \ds\xhdual{\mathcal{A}_h D}{u} \,=\, \xahbimap{D}{u},
    \\[0.75em] 
    \ds\xhdual{D}{ \mathcal{A}_h^* u } \,=\, \xahadjbimap{u}{D}.
    \end{array}\right.
\end{equation}
The operators $\mathcal{A}_h$ and $\mathcal{A}_h^*$ are adjoint with respect to the $L^2(\mathbb{T})$ inner product thanks to \eqref{ah:dual}, that is,
\begin{equation} 
\label{eq:ops_ah_discrete_duality}
    \xhdual{\mathcal{A}_h u}{v} \,=\, a_h(u,v)\,=\, a_h^*(v,u) \,=\,\xhdual{u}{\mathcal{A}_h^*v}\,,
\end{equation}
for all $u \in U_h^m$ and $v \in U_h^m$.
Moreover, $D_{\infty,0}$ lies in the kernel of $\mathcal{A}_h$ by \eqref{ah:kernel}, namely,
\begin{equation} 
\label{eq:ops_ah_discrete_kernel}
    \mathcal{A}_h D_{\infty,0} \,=\, 0 \,.
\end{equation}
Similarly, the discrete operators $ \mathcal{B}_h:V_h\goto W_h$ and $\mathcal{B}_h^*: W_h\goto V_h$ are defined as follows: for all $(v,w) \in V_h \times  W_h$,
\begin{equation} \label{eq:ops_bh_riesz}
\left\{
\begin{array}{l}
	\ds\xhdual{\mathcal{B}_h v}{w} \,=\, \xbhbimap{v}{w},  
    \\[0.75em] 
    \ds\xhdual{v}{\mathcal{B}_h^* w} \,=\, \xbhadjbimap{w}{v}.
    \end{array}\right.
\end{equation}
This formulation naturally preserves the duality structure of the problem
\begin{equation} \label{eq:ops_bh_discrete_duality}
	\xhdual{\mathcal{B}_hv}{w}\,=\,b_h(v,w) =  b^*_h(w,v)\,=\,\xhdual{v}{\mathcal{B}_h^*w}\,,\quad 
	\forall (v,w)\in V_h\times W_h\,,
\end{equation}
where the finite element spaces $(V_h, W_h)$ are defined according to the choice of $(b_h,b_h^*)$ in  Section \ref{sec:disc_poisson}. 

This strong formulation of the system \eqref{eq:vpfp1d_hermite_semidiscrete_vfp}-\eqref{eq:vpfp1d_hermite_semidiscrete_poisson}  is solely intended to simplify the presentation and to make it closer to the continuous framework.


\subsection{Main result and strategy}\label{marker:QY1}

Our main result tackles the large time behavior of the discrete solutions to the Hermite--discontinuous Galerkin numerical scheme \eqref{eq:vpfp1d_hermite_semidiscrete_riesz}-\eqref{eq:ops_bh_riesz} for the fully nonlinear Vlasov-Poisson-Fokker-Planck system \eqref{eq:vpfp1d}. More precisely, Theorem \ref{thm:vpfp1d_hermite_semidiscrete_main_theorem} below  ensures that, under some smallness assumptions on the initial data, the scheme \eqref{eq:vpfp1d_hermite_semidiscrete_riesz}-\eqref{eq:ops_bh_riesz} provides {a structure preserving} approximation in the large time regime, in the sense that the numerical approximation converges to the equilibrium state \eqref{eq:vpfp1d_hermite_equilibrium} with exponential decay rates {uniform in the discretization parameters $h$ and $N_H$}. This reflects the behavior of the continuous solutions to the Vlasov-Poisson-Fokker-Planck system \eqref{eq:vpfp1d}, and, since the exponential rate is uniform in the mesh size, it ensures uniform numerical stability among others.

Solving the system \eqref{eq:vpfp1d_hermite_semidiscrete_riesz}-\eqref{eq:ops_bh_riesz} yields the Hermite coefficients $D_h\,=\,(D_{h,k})_{0\leq k\leq N_H}$, from which we construct the approximation $f_h$ of the solution $f$ to \eqref{eq:vpfp1d} as
\begin{equation}\label{discrete:kinetic}
    f_h(t,x,v) = \sum_{k=0}^{N_H} D_{h,k}(t,x) \,e_k\left(v\right)\,,
\end{equation}
whereas the equilibrium $f_{\infty} = \rho_\infty\mathcal{M}$ is given by \eqref{def:M}. Then, we introduce the following functional for all $t\in \mathbb{R}^+$
\begin{equation}\label{Eh:0}
    \mathcal{E}_h(t)
    \ds\,:=\,\ds
    \frac{1}{2}\norm{f_h(t)-f_{\infty}}_{L^2(f_\infty^{-1})}^2 \,+\, \frac{1}{2T_0}\norm{E_h(t)}_{L^2(\mbT)}^2\,,
\end{equation}
where $L^2(f_\infty^{-1})$ stands for $L^2(\mbT\times\mbR, f_\infty^{-1}\dd{x}\dd{v})$.
Our result reads as follows.
\begin{theorem} \label{thm:vpfp1d_hermite_semidiscrete_main_theorem} Let $\tau_0, T_0>0$ be fixed
    and consider the solution $(D_h,E_h)$ to \eqref{eq:vpfp1d_hermite_semidiscrete_riesz}-\eqref{eq:ops_bh_riesz}, with initial data $D_h^{\textrm{in}}$ satisfying the compatibility condition $\xhdual{D^{\textrm{in}}_{h,0}-D_{\infty,0}}{1}=0$. {Then,} there exists a positive constant $\kappa$ such that, if the initial data satisfies
    \begin{equation} \label{eq:vpfp1d_hermite_semidiscrete_hypocoercive_condition}
    \mathcal{E}_h(t=0)
     \,\le\, \kappa^2\,\min(\tau_0^2,\tau_0^{-2}),
    \end{equation}
    then it holds for all $t\in \mathbb{R}^+$
    \begin{equation} \label{eq:vpfp1d_hermite_semidiscrete_hypocoercive_estimate}
    	\mathcal{E}_h(t)
    	 \,\le\, 3\,\mathcal{E}_h(t=0)\,
    	 e^{-\kappa\,\min(\tau_0,\tau_0^{-1})\,t}\;.
    \end{equation}
The constant $\kappa$ depends only on the temperature $T_0$, the
domain length $\abs{\mbT}$, {the background density
  $\rho_0$,} the polynomial degree $m$, and the mesh quasi-uniformity constant $C_{qu}$. 
\end{theorem}
Let us comment on Theorem \ref{thm:vpfp1d_hermite_semidiscrete_main_theorem}. It guarantees that, under the smallness assumption \eqref{eq:vpfp1d_hermite_semidiscrete_hypocoercive_condition}, the approximation $f_h$ obtained by solving \eqref{eq:vpfp1d_hermite_semidiscrete_riesz}-\eqref{eq:ops_bh_riesz} relaxes toward the equilibrium $f_\infty$. We obtain exponential relaxation rates which are uniform in the mesh parameters $h$ and $N_H$. \\
\label{marker:QY1b}{A byproduct of our analysis is
  that, under the smallness assumption
  \eqref{eq:vpfp1d_hermite_semidiscrete_hypocoercive_condition}, the
  system
  \eqref{eq:vpfp1d_hermite_semidiscrete_riesz}-\eqref{eq:ops_bh_riesz}
  has a unique global-in-time solution. Indeed,
  \eqref{eq:vpfp1d_hermite_semidiscrete_riesz}-\eqref{eq:ops_bh_riesz}
  interpreted as a nonlinear ordinary differential equation in finite
  dimension has locally Lipschitz coefficients thanks to Lemmas
  \ref{lem:ops_bh_props} and \ref{lem:ops_ah_props}
  below. In fact, the right-hand side is quadratic  in $D_h$ once the discrete Poisson problem \eqref{eq:ops_bh_elliptic}
  is solved for $E_h$. Therefore, it
  has a local solution which is unique on its existence interval
  thanks to the Cauchy-Lipschitz theorem. The estimate
  \eqref{eq:vpfp1d_hermite_semidiscrete_hypocoercive_estimate} and the
  equivalence of norms in finite dimension guarantee that this
  solution is in fact global in time.}
\\
{In principle, our analysis yields an explicit
  rate $\kappa$ in terms of $(T_0,\abs{\mbT},\rho_0,m,C_{qu})$,
    but we choose not to track this dependence explicitly in the
  proof to keep the computations tractable.}
\\

Before proceeding to the proof, we detail our strategy to derive \eqref{eq:vpfp1d_hermite_semidiscrete_hypocoercive_estimate}. There are two main steps in order to prove the exponential convergence of $f_h$, solution to the discrete nonlinear Vlasov-Poisson-Fokker-Planck system \eqref{eq:vpfp1d_hermite_semidiscrete_riesz}-\eqref{eq:ops_bh_riesz}, towards the equilibrium state $f_{\infty}$. The first step consists in analyzing the large time behavior of the discrete approximations to the system \eqref{eq:vpfp1d} linearized near $f_\infty$. These approximations are computed by solving a linearized scheme, corresponding to \eqref{eq:vpfp1d_hermite_semidiscrete_riesz} without the nonlinear terms $\Pi_h\left( E_h(D_{h,0}-D_{\infty,0})\right)$ and $\Pi_h\left( E_hD_{h,k-1}\right)$, $k\geq 2$. This scheme reads as follows
\begin{equation} \label{eq:vpfp1d_hermite_semidiscrete_riesz:linz}
	\left\{
	\begin{array}{l}
		\ds \partial_t D_{h,0} \;-\; \mathcal{A}_h^*D_{h,1} \;=\; 0,
		\\ [0.6em]
		\ds \partial_t D_{h,1} \;+\; \mathcal{A}_h D_{h,0} - \sqrt{2}\;\mathcal{A}_h^*D_{h,2} \;+\; \frac{D_{\infty,0}}{T_0}\;\mathcal{A}_h\Phi_h  \,=\, -\frac{1}{\tau_0}D_{h,1},
		\\ [0.9em]
		\ds \partial_t D_{h,k} \;+\; \sqrt{k}\;\mathcal{A}_hD_{h,k-1} \;-\; \sqrt{k+1} \;\mathcal{A}_h^*D_{h,k+1}  \,=\, -\frac{k}{\tau_0}\;D_{h,k} \qc 2\leq k \leq N_H\,,
		\\ [1.2em]
		\ds E_h \;=\; -\mathcal{B}_h\Phi_h \qc -\mathcal{B}_h^*E_h \;=\; D_{h,0}-D_{\infty,0}\qc     \xhdual{\Phi_h}{1} \;=\; 0\,,
		\\ [1.em]
		\ds D_h(t=0)\,=\, D_{h}^{\textrm{in}}\,,
	\end{array}
	\right.
\end{equation}
with $D_{h,N_H+1}=0$ and where $(\mathcal{A}_h,\mathcal{A}^*_h,\mathcal{B}_h,\mathcal{B}^*_h)$ are defined by \eqref{eq:ops_ah_riesz}-\eqref{eq:ops_bh_riesz}. The key point is that for the linearized system \eqref{eq:vpfp1d_hermite_semidiscrete_riesz:linz}, the functional $\mathcal{E}_h$ given in \eqref{Eh:0} constitutes a Lyapunov functional. It is indeed dissipated according to the following statement.
\begin{proposition}\label{prop:4:2}
	Consider the solution $(D_h,E_h)$ to \eqref{eq:vpfp1d_hermite_semidiscrete_riesz:linz} along with its associated linearized energy $\mathcal{E}_h$ defined by \eqref{Eh:0}. Then the following linearized energy estimate holds for all $t\in \mathbb{R}^+$,  
	\[
	\frac{\dd}{\dd{t}}\mathcal{E}_h(t) \,+\, \frac{1}{\tau_0}\,\mathcal{I}_h(t)
	\,=\,0
	\,,
	\]
	where $\mathcal{I}_h$ is defined for all $t\geq 0$ as
	\begin{equation}\label{def:I:h}
	\mathcal{I}_h(t) \;\defeq\; 
    \frac{1}{D_{\infty,0}}\sum_{k=0}^{N_H} k\;\xnorm{D_{h,k}(t) - D_{\infty,k}}^2.
	\end{equation}
\end{proposition}
\begin{proof}
We first rewrite the linearized energy functional $\mathcal{E}_h$ within the Hermite framework. The orthogonality property of the Hermite functions $(e_k)_{k\in\mbN}$ in $L^2(\mathcal{M}^{-1}\dd{v})$ yields the following identity
	\begin{equation}
		\ds\mathcal{E}_h(t)
		\ds\,=\,
		\frac{1}{2\,D_{\infty,0}}\sum_{k=0}^{N_H}\xnorm{D_{h,k}(t)-D_{\infty,k}}^2 \,+\, \frac{1}{2\,T_0}\xnorm{E_h(t)}^2\,.\label{Eh}
	\end{equation}
To compute the time derivative of the previous right hand side, we sum
the $L^2$-inner products of the first, second and third lines in
\eqref{eq:vpfp1d_hermite_semidiscrete_riesz:linz} with $ D_{h,0} -
D_{\infty,0}$, $ D_{h,1}$ and $ D_{h,k}$, $k\in \{2, \dots, N_H\}$ respectively, divided by $D_{\infty,0}$, which yields
	\[
	\frac{\dd\mathcal{E}_h}{\dd{t}} \,+\, \frac{1}{\tau_0}\mathcal{I}_h \,+\, \mathcal{C}_h\,+\, \mathcal{D}_h \,=\, 0,
\]
where $\mathcal{I}_h$ is defined in \eqref{def:I:h}
whereas $\mathcal{C}_h$ and $\mathcal{D}_h$ are given by
$$
		\left\{
		\begin{array}{l}
			\ds \mathcal{C}_h \defeq
            \frac{1}{D_{\infty,0}}
            \left(
            \xhdual{\mathcal{A}_h^*D_{h,1}}{D_{\infty,0}} + \sum_{k=0}^{N_H}\sqrt{k+1}\,\left(\xhdual{\mathcal{A}_h D_{h,k}}{D_{h,k+1}} -\xhdual{\mathcal{A}_h^* D_{h,k+1}}{D_{h,k}} \right)
            \right), 
			\\ [1.2em]
			\ds \mathcal{D}_h \defeq 
			\frac{1}{T_0}\,\xhdual{\mathcal{A}_h\Phi_h}{D_{h,1}}
            -
            \frac{1}{T_0} \xhdual{E_h}{\partial_t E_h}. 
		\end{array}
		\right.
        $$	
On the one hand, applying the duality property \eqref{eq:ops_ah_discrete_duality} and the kernel property \eqref{eq:ops_ah_discrete_kernel}, we obtain that $\mathcal{C}_h = 0$.  On the other hand, to show that $\mathcal{D}_h = 0$, we reformulate its first term thanks to the duality properties \eqref{eq:ops_ah_discrete_duality} and \eqref{eq:ops_bh_discrete_duality} along with the first and last lines in \eqref{eq:vpfp1d_hermite_semidiscrete_riesz:linz}, leading to
	\begin{equation*}\label{marker:Q1X}
		\begin{aligned}
			\xhdual{\mathcal{A}_h\Phi_h}{D_{h,1}}
			&=  \xhdual{\Phi_h}{\mathcal{A}_h^*D_{h,1}} = \xhdual{\Phi_h}{\partial_t(D_{h,0}-D_{\infty,0})} \\[0.8em]
			&= -\xhdual{\Phi_h}{\partial_t\,\mathcal{B}_h^*E_h} = -\xhdual{\mathcal{B}_h\Phi_h}{\partial_t E_h}= \xhdual{E_h}{\partial_t E_h}.   
		\end{aligned}
	\end{equation*}
	{The third equality is obtained by taking the time derivative of the second line in \eqref{eq:vpfp1d_hermite_semidiscrete_poisson}, corresponding to the weak formulation of the Poisson equation, which leads to
	\[
	-\xbhadjbimap{\partial_t E_h}{v} = \xhdual{\partial_t(D_{h,0}-D_{\infty,0})}{v},
	\]
	for all $v\in V_h$ and where the equality holds pointwise in $t\geq 0$, and then taking for each $t\geq 0$ the test function $v\,=\, \Phi_h(t)$.
	}
\end{proof}
Unfortunately, the estimate in Proposition \ref{prop:4:2} itself is not sufficient to prove the convergence of
the solution to the linearized scheme \eqref{eq:vpfp1d_hermite_semidiscrete_riesz:linz} towards the stationary state $f_{\infty}$ since the dissipation rate $\mathcal{I}_h$ fails to control the energy functional $\mathcal{E}_h$, that is,
\[
\mathcal{E}_h(t)\not\lesssim\mathcal{I}_h(t)  \,.
\]
The previous relation may be checked by comparing \eqref{def:I:h} and \eqref{Eh} for all $D_h=(D_{h,k})_{0\leq k\leq N_H}$ such that $D_{h,k}\,=\,0$ for $k\geq 1$ and $D_{h,0}\,\neq\,D_{\infty,0}$.
This lack of coercivity of the dissipation is precisely the situation addressed by hypocoercivity methods. To bypass this difficulty, we define a modified relative entropy $\mathcal{H}_h$ as
\begin{equation}
	\mathcal{H}_h(t) \,=\,
	\mathcal{E}_h(t)\,-\, \frac{\alpha_0}{D_{\infty,0}}\,\left\langle
	D_{h,1}(t),\, F_h(t)\right\rangle_{L^2\left(\mathbb{T}\right)}\,,
	\label{eq:H0}
\end{equation}
where $\alpha_0>0$ is a small free parameter and where $F_h$ is computed by solving the following elliptic problem: find
 $(F_h,\Psi_h) \in U_h^m\times U_h^m$ such that $  \xhdual{\Psi_h}{1} = 0$ and
\begin{equation} \label{eq:ops_ah_elliptic}
	F_h = -\mathcal{A}_h \Psi_h  \qc -\mathcal{A}_h^* F_h  = D_{h,0}-D_{\infty,0}\,.
\end{equation}
The key point here is that this additional term will provide the missing dissipative term of order $$
-\frac{\alpha_0}{D_{\infty,0}}\,\|D_{h,0}- D_{\infty,0}\|_{L^2(\mbT)}^2.
$$ 
Hence, to get the convergence of the solution to the linearized system
\eqref{eq:vpfp1d_hermite_semidiscrete_riesz:linz} towards the stationary state, the strategy consists in
proving that $\mathcal{H}_h$ and $\mathcal{E}_h$ induce equivalent norms (see Lemma \ref{lem:vpfp1d_hermite_semidiscrete_free_energy_equivalence} below) and that there exists a
constant  $\kappa>0$, independent of $h$, $N_H$ and $\tau_0$, such that
$$
\frac{\dd{}}{\dd{t}} \mathcal{H}_h(t) \leq - \kappa\, \min(\tau_0,\tau_0^{-1})\,\mathcal{H}_h(t).
$$
This strategy has already been implemented in \cite{blaustein_discrete_2024, blaustein_structure_2024, blaustein_structure_2025} within a finite volume framework for the spatial variable applied to the linearized scheme \eqref{eq:vpfp1d_hermite_semidiscrete_riesz:linz}. In the present article, we extend this strategy to the discontinuous Galerkin framework and to the fully nonlinear system \eqref{eq:vpfp1d}. 
The main additional difficulty is due to the fact that $\mathcal{E}_h$ is a Lyapunov functional neither for the continuous nonlinear model nor for the fully nonlinear scheme \eqref{eq:vpfp1d_hermite_semidiscrete_riesz}-\eqref{eq:ops_bh_riesz}. Hence, we are led to control the additional nonlinear terms of order $\mathcal{E}_h(t)^2$ leading to the smallness assumption \eqref{eq:vpfp1d_hermite_semidiscrete_hypocoercive_condition} in Theorem \ref{thm:vpfp1d_hermite_semidiscrete_main_theorem}. The proof of this result is provided in Section \ref{sec:proof_of_main_theorem} below. In Section \ref{sec:preliminary_results}, we present several preliminary results on the structure-preserving properties of the  operators $(\mathcal{A}_h,\mathcal{A}_h^*)$ and $(\mathcal{B}_h,\mathcal{B}_h^*)$ (see Lemmas \ref{lem:ops_ah_props}, \ref{lem:ops_ah_discrete_regularity_estimate} and \ref{lem:ops_bh_props} below).
\subsection{Discrete properties}\label{marker:QY2}
Before moving on to the heart of the proof of Theorem \ref{thm:vpfp1d_hermite_semidiscrete_main_theorem}, we mention that the numerical method \eqref{eq:vpfp1d_hermite_semidiscrete_riesz}-\eqref{eq:ops_bh_riesz} preserves some physical properties of the nonlinear Vlasov-Poisson-Fokker-Planck system \eqref{eq:vpfp1d}. For a distribution function $f$ given by the Hermite expansion \eqref{eq:hermite_decomposition}, we define  the global mass $m_0$ by
$$
m_0(t) \,:=\, \int_{\mathbb{T}\times\mathbb{R}} f(t,x,v)\,\dd{x}\dd{v}\,=\, \int_{\mathbb{T}} D_{0}(t,x)\,\dd{x}, \quad \forall t\geq 0\;,
$$
whereas the total momentum $m_1$ is given by
$$
m_1(t) \,:=\, \int_{\mathbb{T}\times\mathbb{R}} v\, f(t,x,v)\,\dd{x}\dd{v}\,=\, \sqrt{T_0}\;\int_{\mathbb{T}} D_{1}(t,x)\,\dd{x}\,,\quad \forall t\geq 0\,.
$$
Finally, the total energy $W$ is defined as
\begin{equation*}
W(t) \;:=\; K(t) \;+\, 
        \frac{1}{2}\,\norm{E(t)}_{L^2(\mbT)}^2\,,\quad \forall t\geq 0,
\end{equation*}
with the kinetic energy $K$
$$
K(t) \,:=\,\int_{\mathbb{T}\times\mathbb{R}}
        \frac{|v|^2}{2}\;f(t,x,v)\,\dd{x}\dd{v}
       \,=\, \frac{T_0}{\sqrt 2}\int_{\mathbb{T}} \left(D_{2} + \frac{1}{\sqrt 2} D_{0}\right)(t,x)\,\dd{x}\,,\quad \forall t\geq 0.
$$
Then for a distribution function $f$ solution to the Vlasov-Poisson-Fokker-Planck system \eqref{eq:vpfp1d}, we easily demonstrate  the following properties
\begin{equation}\label{C0:inv}
\frac{\dd}{\dd t}\left(\begin{array}{l} m_0 \\ m_1 \\ W\end{array}\right)(t) \,=\,  -\frac{1}{\tau_0}\,\left(\begin{array}{l} \quad 0 \\ \quad m_1 \\ 2K - T_0\; m_0\end{array}\right)(t).
\end{equation}
Observe that in the collisionless regime $\tau_0=+\infty$, it corresponds to the conservation of mass, momentum and energy for the Vlasov-Poisson system. 

Now, let us see how these properties are preserved by the various discretization methods we previously proposed. {The discrete analogs of mass, momentum and kinetic energy of the solution  $D_h\,=\,(D_{h,k})_{0\leq k\leq N_H}$ to \eqref{eq:vpfp1d_hermite_semidiscrete_riesz}-\eqref{eq:ops_bh_riesz} are given by
\begin{equation*}\label{def:disc:mass:momentum:energy}
\left(\begin{array}{l} m_{h,0} \\[0.2em]  m_{h,1} \\[0.2em]  K_h\end{array}\right)(t)\,:=\,
\int_{\mathbb{T}} \left(\begin{array}{l} \quad \quad \;\;\;\;\;\ds D_{h,0} \\[0.2em] \;\;\;\;\;\;\;\ds \sqrt{T_0} D_{h,1} \\[0.em]
\ds \frac{T_0}{\sqrt 2}(D_{h,2} + \frac{1}{\sqrt 2} D_{h,0})
\end{array}\right)(t,x)\,\dd{x}
\;,
\end{equation*}
for all $t\geq 0$, and the discrete total energy is defined as 
\begin{equation*}
	W_h(t) \;:=\; K_h(t) \;+\, 
	\frac{1}{2}\,\norm{E_h(t)}_{L^2(\mbT)}^2\,,\quad \forall t\geq 0.
\end{equation*}
}
We first study the case of fully discontinuous Galerkin methods \eqref{eq:vpfp1d_hermite_semidiscrete_riesz}-\eqref{eq:ops_bh_riesz} with $b_h$ given by \eqref{eq:ops_bh_ldg_global}.

\begin{proposition} 
\label{prop1:disc:inv} Let $T_0>0$ and $\tau_0>0$ be fixed  in \eqref{eq:vpfp1d_hermite_semidiscrete_riesz}. The solution $(D_h,E_h)$ to \eqref{eq:vpfp1d_hermite_semidiscrete_riesz}-\eqref{eq:ops_bh_riesz} with $b_h$ given by \eqref{eq:ops_bh_ldg_global} possesses the following discrete properties:
\begin{enumerate}[label=(\roman*), wide]
    \item conservation of the total mass 
    $$
    m_{h,0}(t)\,=\,m_{h,0}(0)\,,\quad \forall\, t\,\geq\, 0\,;
    $$
    \item\label{disc:momentum} {consistent total momentum estimate: when $b_h$  is given by \eqref{eq:ops_bh_ldg_global} with $\hat{g}_b(\Phi) \,=\, \Phi^-$, it holds
    \begin{equation*} 
    \frac{\dd m_{h,1}}{\dd t}(t)\,+\,\mathcal{U}_h(t) \,=\, -\frac{m_{h,1}}{\tau_0}(t)\,,
        \quad \forall\, t\in \mathbb{R}^+\,,
    \end{equation*}
    and when $b_h$  is given by \eqref{eq:ops_bh_ldg_global} with $\hat{g}_b(\Phi) \,=\, \Phi^+$, it holds
    \begin{equation*} 
    	\frac{\dd m_{h,1}}{\dd t}(t) \,-\,\mathcal{U}_h(t)\,=\, -\frac{m_{h,1}}{\tau_0}(t)\,,
    	\quad \forall\, t\in \mathbb{R}^+\,,
    \end{equation*}
    where $\mathcal{U}_h(t)$ is the nonnegative numerical remainder given by
   \[
   \mathcal{U}_h(t)\,=\,\frac{1}{2}\;\sum_{j\in\mcJ}
   \jmp{E_h(t)}_{j-\half}^2 \,\geq\,0\,;
   \] 
}
    \item\label{disc:energy} preservation of the energy structure: with 
    $ (\hat{g}_a, \hat{g}_a^*)= (\hat{g}_b, \hat{g}_b^*) $  in \eqref{eq:ops_ah_flux} and \eqref{eq:ops_bh_ldg_flux}, it holds
    \begin{equation*} 
    \frac{\dd W_h}{\dd t}(t) \,=\, -\frac{1}{\tau_0} \,\left( 2\, K_{h} \,-\; T_0\, m_{h,0}\right)(t)\,,\quad \forall t\in \mathbb{R}^+\,.
    \end{equation*}
\end{enumerate} 
\end{proposition}
We observe that the time evolution of the total mass and energy coincides exactly with the continuous case. However, {the discrete total momentum estimate is preserved up to the remainder $\pm\,\mathcal{U}_h(t)$, whose sign depends on the flux choice and which is consistent with $0$ as $h\rightarrow 0$ (since for smooth solutions, $\jmp{E_h}_{j-\half}$ is small with respect to $h$).} The proof of this property is detailed in Appendix A.

Then, we provide slightly different results when the discontinuous Galerkin method \eqref{eq:vpfp1d_hermite_semidiscrete_riesz}-\eqref{eq:ops_bh_riesz} is coupled with the Raviart-Thomas method \eqref{eq:ops_bh_rt} for the discretization of the Poisson equation. 

\begin{proposition} 
\label{prop2:disc:inv} Let $T_0>0$ and $\tau_0>0$ be fixed  in \eqref{eq:vpfp1d_hermite_semidiscrete_riesz}. The solution $(D_h,E_h)$ to \eqref{eq:vpfp1d_hermite_semidiscrete_riesz}-\eqref{eq:ops_bh_riesz} with $b_h$ given by \eqref{eq:ops_bh_rt}  possesses the following discrete properties:
\begin{enumerate}[label=(\roman*), wide]
    \item conservation of the total mass 
    $$
    m_{h,0}(t)\,=\,m_{h,0}(0)\,,\quad \forall\, t\,\geq\, 0\,;
    $$
    \item 
    	exact total momentum estimate
    \begin{equation*} 
    \frac{\dd m_{h,1}}{\dd t}(t) \,=\, -\frac{m_{h,1}}{\tau_0}(t)\,,
        \quad \forall\, t\in \mathbb{R}^+\,.
    \end{equation*}
\end{enumerate}
\end{proposition}
In the Raviart-Thomas case, the mass and momentum estimates  are identical to the continuous case, but the energy structure is not preserved.

The proof of these properties is detailed in Appendix A.

\subsection{Preliminary results}
\label{sec:preliminary_results}

To investigate the structural properties of the discrete operators arising in the semi-discrete scheme \eqref{eq:vpfp1d_hermite_semidiscrete_riesz}, we introduce the standard discontinuous Galerkin functional setting. Instead of the classical Sobolev space $H^1(\mbT)$, we consider the broken Sobolev space $H^1(\mcT_h)$ defined by 
\begin{equation*} 
    H^1(\mcT_h) \,\defeq\; \set{u \in L^2(\mbT) \setmid \restr{u}{K_j} \in H^1(K_j) \qc \forall j \in \mcJ},
\end{equation*}
which is equipped with the following seminorm:
\begin{equation*} 
    \xdhseminorm{u} \,\defeq\, \left(\xhnorm{\partial_hu}^2 + \xjhseminorm{u}^2\right)^{\frac{1}{2}},
\end{equation*}
where $\partial_h u$ denotes the broken derivative of $u$ and is defined by
\begin{equation}\label{eq:broken_derivative}
    \restr{(\partial_h u)}{K_j} \,\defeq\, \partial_xu \qc \forall j \in \mcJ,
\end{equation} 
and the jump seminorm $\xjhseminorm{\cdot}$ is given by
\begin{equation*}
    \xjhseminorm{u} \,=\, \left(\sum_{j\in\mcJ} h_{j-\half}^{-1}\jmp{u}_{j-\half}^2\right)^{\half} \qc h_{j-\half} = \min(h_{j-1},h_j).
\end{equation*}
The seminorm $\xdhseminorm{\cdot}$ induces a norm on the subspace of $H^1(\mcT_h)$ consisting of functions satisfying $\xhdual{u}{1} = 0$, in which we will establish the coercivity of the discrete operators. {This is a consequence of the Poincaré-Wirtinger inequality for piecewise $H^1$ functions (see \cite[Remark $1.1$]{Brenner04} for instance), which reads
\begin{equation}\label{Poinc:piecewise:H1}
	\xhnorm{u} \,\lesssim\, \xdhseminorm{u}\,,
\end{equation}
for any $u \in U_h^m$ satisfying $\xhdual{u}{1} = 0$ and where the implicit constant depends on $C_{qu}$ and $|\mathbb{T}|$ only.
} \\

We then introduce {the standard trace inequalities (see \cite{dipietro_mathematical_2012} for instance)}, which are repeatedly used throughout our analysis. There exists a positive constant $C_{tr} = C_{tr}(m)$ such that, for all $j \in \mcJ$ and $u \in \mathscr{P}_m(K_j)$, 
\begin{equation} \label{eq:trace_inequality}
    \xKnorm[\partial K_j]{u} \le C_{tr} h_j^{-\half} \xKnorm[K_j]{u} \qc \xKnorm[\partial K_j]{u}^2 = |u(x_{j-\half}^+)|^2 + |u(x_{j+\half}^-)|^2.
\end{equation}
In particular, one may take $C_{tr}(m) = \sqrt{(m+1)(m+2)}$.

We now turn to the key properties of the discrete operators $\mathcal{A}_h$ and $\mathcal{A}_h^*$, collected in Lemmas \ref{lem:ops_ah_props} and \ref{lem:ops_ah_discrete_regularity_estimate}. For clarity of notation, we write $A \lesssim B$ to mean that $A \le CB$ for some $C > 0$ independent of the discretization parameters $h$ and $N_H$ and of $\tau_0$. Moreover, $A \sim B$ means both $A \lesssim B$ and $B \lesssim A$.
\begin{lemma} \label{lem:ops_ah_props}
    Let $(\mathcal{A}_h,\mathcal{A}_h^*)$ be the discrete operators associated with the bilinear forms $(a_h,a_h^*)$, defined by \eqref{eq:ops_ah_global} with alternating fluxes.
    It holds:
    \begin{enumerate}[label=(\roman*), wide]
        \item \label{enum:ops_ah_discrete_primal_dual_balance} preservation of the primal-dual balance. For all $u \in U_h^m$, we have
        \begin{equation} \label{eq:ops_ah_discrete_primal_dual_balance}
        	\xhnorm{\mathcal{A}_hu} \sim \xhnorm{\mathcal{A}_h^*u}\sim\xdhseminorm{u}\,,
        \end{equation}
        \item \label{enum:ops_ah_discrete_coercivity} preservation of the coercivity. For all $u \in U_h^m$ such that $\xhdual{u}{1} = 0$, we have
        \begin{equation} \label{eq:ops_ah_discrete_coercivity}
               \|u\|_{L^2(\mathbb{T})} \lesssim \xhnorm{\mathcal{A}_h u}\,.
                    \end{equation}
        The implicit constants depend only on $T_0$, the degree $m$, the mesh parameter $C_{qu}$ and the domain length $|\mathbb{T}|$. 
    \end{enumerate}
\end{lemma}
\begin{proof} \label{proof:ops_ah_discrete_duality_and_kernel}
We first prove item \ref{enum:ops_ah_discrete_primal_dual_balance} in the case where $\hat{g}_a(u) = u^-$ in \eqref{eq:ops_ah_global}.  On the one hand,  using the trace inequality \eqref{eq:trace_inequality} and the Cauchy-Schwarz inequality, we get that for any $u\in U_h^m$,
\begin{equation*}
	\begin{aligned}
		\xhnorm{\mathcal{A}_hu}^2 \,=\,a_h(u, \mathcal{A}_hu)
        \,=\,a^*_h(\mathcal{A}_hu, u)
		&= \sqrt{T_0}\left(\sum_{j\in\mcJ} \xKdual[K_j]{\partial_x u}{\mathcal{A}_hu} + \sum_{j\in\mcJ} \jmp{u}_{j-\half}(\mathcal{A}_hu)_{j-\half}^+\right)\\[-1.em]
		\\ 
		&\le \sqrt{T_0}\,\left(\xhnorm{\partial_h u} \,+\, C_{qu}\, C_{tr}(m)\,\xjhseminorm{u}\right)\,\xhnorm{\mathcal{A}_hu}
		\\[.5em]
		&\le \sqrt{2\,T_0}\,\max\left(1,C_{qu} \,C_{tr}(m)\right)\,\xdhseminorm{u}\,\xhnorm{\mathcal{A}_hu}\,,
	\end{aligned}
\end{equation*}
which yields the first inequality
$$
\xhnorm{\mathcal{A}_hu} \,\lesssim \,\xdhseminorm{u}.
$$
On the other hand, for $u\in U_h^m$, we show the reverse estimate $\xdhseminorm{u} \lesssim \xhnorm{\mathcal{A}_hu} $ following the lines of \cite[Lemma $2.4$]{wang_stability_2015}. Let $L_m(\hat{x})$ be the standard Legendre polynomial of degree $m$ on $\hat{K} = [-1, 1]$. We define the local interpolant $\pi_j^+$ for any $j \in \mcJ$ as follows:  for a polynomial $v\in\mathscr{P}_m(K_j)$, $\pi_j^+ v $ is given by
\begin{equation*}
    \pi_j^+ v(x) = (-1)^m v(x_{j-\half}^+)\,L_m(\hat{x}),
\end{equation*}
with $\hat{x} = 2(x-x_j)/h_j$ and $x_j = (x_{j-\half} + x_{j+\half})/2$. Using the trace inequality \eqref{eq:trace_inequality}, we obtain 
\begin{equation}
\label{estim:pi:j}
    \xKnorm[K_j]{\pi_j^+ v} \,=\, \sqrt{\frac{h_j}{2}}\,\xKnorm[\hat{K}]{L_m}\,|v(x_{j-\half}^+)| \,\le\, C_{ip}(m)\,\xKnorm[K_j]{v} \,,
\end{equation}
where $C_{ip}(m) = C_{tr}(m)/\sqrt{2m+1}$. We then define the test function $r\in U^m_h$ for a function $u \in U_h^m$ as 
\[
r(x)\,=\, \left\{
\begin{array}{l}
	 \left(\partial_xu - \pi_j^+\partial_xu\right)(x)\,, \,\textrm{ for }\,x\in K_j\,,\\[0.8em]
	0,\,\textrm{ elsewhere}\,.
\end{array}
\right.
\]
Since $L_m(-1) = (-1)^m$ and since $\xKdual[K_j]{\partial_x u}{     \pi_j^+\partial_x u} \,=\, 0$ due to the orthogonality property of Legendre polynomials, it holds
\begin{equation*}
    r(x_{j-\half}^+) = 0 \qand \xKdual[K_j]{\partial_x u}{r} = \xKnorm[K_j]{\partial_xu}^2.
\end{equation*}
Furthermore, since $r$ is supported on $K_j$ with $r(x_{j-\half}^+)=0$, using \eqref{eq:ops_ah_riesz}, the duality \eqref{ah:dual} and \eqref{eq:ops_ah_global}, we have
\[
\xKdual[K_j]{\mathcal{A}_hu}{r} \,=\,
\xKdual[\mathbb{T}]{\mathcal{A}_hu}{r}\,=\,
a_h(u,r)\,=\,a^*_h(r,u)\,=\,
\sqrt{T_0}\xKdual[K_j]{\partial_x u}{r} = \sqrt{T_0}\xKnorm[K_j]{\partial_xu}^2\,.
\]
We apply the Cauchy-Schwarz inequality in the left hand side of the previous relation and bound $\|r\|_{L^2(K_j)}$ thanks to \eqref{estim:pi:j}
\begin{equation*}
\begin{aligned}
    \sqrt{T_0}\xKnorm[K_j]{\partial_x u}^2 = \xKdual[K_j]{\mathcal{A}_hu}{r} 
    &\le \left(1+C_{ip}(m)\right)\xKnorm[K_j]{\mathcal{A}_hu}\xKnorm[K_j]{\partial_xu}.
\end{aligned}
\end{equation*}
This implies 
\begin{equation}\label{interm:estimate}
		\xKnorm[K_j]{\partial_x u}
		\le \sqrt{\frac{1}{T_0}}\left(1+C_{ip}(m)\right)\xKnorm[K_j]{\mathcal{A}_hu}.
\end{equation}
In addition, we evaluate $a_h(u, 1_{K_j})$ according to \eqref{eq:ops_ah_global} where $1_{K_j}$ is the indicator function of $K_j$ and $\hat{g}_a(u) = u^-$ to get
\begin{equation*}
    \jmp{u}_{j-\half} = \sqrt{\frac{1}{T_0}}\xKdual[K_j]{\mathcal{A}_hu}{1} - \xKdual[K_j]{\partial_x u}{1}. 
\end{equation*}
By the Cauchy-Schwarz inequality and \eqref{interm:estimate}, we deduce
\begin{equation}
\label{jumpy}
    \left|\jmp{u}_{j-\half}\right| \le h_j^{\half}\left(\sqrt{\frac{1}{T_0}}\xKnorm[K_j]{\mathcal{A}_hu} + \xKnorm[K_j]{\partial_x u}\right) \le \sqrt{\frac{C_{qu}}{T_0}}\,h_{j-\half}^{\half}\,(2 + C_{ip}(m))\xKnorm[K_j]{\mathcal{A}_hu} .
\end{equation}
Finally, we sum the square of \eqref{interm:estimate} and \eqref{jumpy} over $j \in \mcJ$, which yields
\begin{equation*}
    \xdhseminorm{u} = \left(\sum_{j\in\mcJ}\xKnorm[K_j]{\partial_x u}^2 + \sum_{j\in\mcJ}h_{j-\half}^{-1}\jmp{u}_{j-\half}^2\right)^{\half} \le \frac{C}{\sqrt{T_0}}\xhnorm{\mathcal{A}_hu}\,,
\end{equation*}
where $C>0$ depends on $C_{qu}$ and $C_{ip}(m)$.
This implies that $\xhnorm{\mathcal{A}_hu}\sim \xdhseminorm{u}$ in the case where $\hat{g}_a(u) = u^-$. 
We obtain the same estimate in the case where $\hat{g}_a(u) = u^+$ following the same computations.

\vspace{1em}

Now, for item \ref{enum:ops_ah_discrete_coercivity}, we apply the Poincar{\'e}-Wirtinger inequality {\eqref{Poinc:piecewise:H1} for piecewise $H^1$ functions}, valid for any $u \in U_h^m$ satisfying $\xhdual{u}{1} = 0$ and then item \ref{enum:ops_ah_discrete_primal_dual_balance}, which yields 
\begin{equation*}\label{marker:QX6}
		\xhnorm{u} \,\lesssim\, \xdhseminorm{u} \,\lesssim\, \xhnorm{\mathcal{A}_hu} \,.
\end{equation*}
\end{proof}

Thanks to Lemma \ref{lem:ops_ah_props}, we prove that the solutions to the elliptic problem \eqref{eq:ops_ah_elliptic} associated to the modified relative entropy $\mathcal{H}_h$ in \eqref{eq:H0} enjoy the following properties.
\begin{lemma}\label{lem:ops_ah_discrete_regularity_estimate}
There exists a unique solution 
$(F_h,\Psi_h) \in U_h^m\times U_h^m$ to \eqref{eq:ops_ah_elliptic} with $\xhdual{\Psi_h}{1}=0$, which satisfies:
\begin{enumerate}[label=(\roman*), wide]
    \item \label{enum:discrete_H1_regularity_estimate} discrete $H^1$ elliptic regularity. 
    \begin{equation} \label{eq:ops_ah_discrete_H1_regularity_estimate}
        \xhnorm{F_h} \lesssim \xhnorm{D_{h,0}-D_{\infty,0}},
    \end{equation}
    \item \label{enum:discrete_H2_regularity_estimate} discrete $H^2$ elliptic regularity.
    \begin{equation} \label{eq:ops_ah_discrete_H2_regularity_estimate}
        \xhnorm{\mathcal{A}_h F_h} \lesssim \xhnorm{D_{h,0}-D_{\infty,0}}.
    \end{equation}
\end{enumerate}
The implicit constants in the estimates above depend only on the temperature $T_0$, the domain length $\abs{\mbT}$, the polynomial degree $m$, and the mesh quasi-uniformity constant $C_{qu}$. 
\end{lemma}
\begin{proof}
	On the one hand, the Lax-Milgram theorem applies here thanks
        to the coercivity property
        \eqref{eq:ops_ah_discrete_coercivity}, applied to the
          bilinear form
          $(\Psi,\varphi)\mapsto\xhdual{\mathcal{A}_h\Psi}{\mathcal{A}_h\varphi}$
          on the subspace of $U_h^m$ of zero-mean functions. The
          resulting variational problem is equivalent to
          \eqref{eq:ops_ah_elliptic} because
          $\xhdual{D_{h,0}-D_{\infty,0}}{1}=0$, due to
          the conservation of mass,
          and  by \eqref{eq:ops_ah_discrete_kernel}, we have
          $$
          \xhdual{\mathcal{A}_h^*\mathcal{A}_h\Psi}{1}=\xhdual{\mathcal{A}_h\Psi}{\mathcal{A}_h1}=0.
          $$
Hence, \eqref{eq:ops_ah_elliptic} admits a unique solution $(F_h,\Psi_h) \in U_h^m\times U_h^m$ such that $\xhdual{\Psi_h}{1}=0$.\\
On the other hand, the duality property in \eqref{eq:ops_ah_discrete_duality}, the Cauchy-Schwarz inequality and the coercivity property \eqref{eq:ops_ah_discrete_coercivity} in Lemma \ref{lem:ops_ah_props} yield item \ref{enum:discrete_H1_regularity_estimate}, that is, 
\begin{eqnarray*}
    \xhnorm{F_h}^2  &=& \xhdual{\mathcal{A}_h^*\mathcal{A}_h\Psi_h}{\Psi_h} \\[0.5em]
    &=& \xhdual{D_{h,0}-D_{\infty,0}}{\Psi_h}\\[0.5em]
    & \le& \xhnorm{D_{h,0}-D_{\infty,0}}\xhnorm{\Psi_h} \\[0.5em]
&\lesssim& \xhnorm{D_{h,0}-D_{\infty,0}} \xhnorm{F_h }.     
\end{eqnarray*}

We next prove item \ref{enum:discrete_H2_regularity_estimate} by applying \eqref{eq:ops_ah_discrete_primal_dual_balance} in Lemma \ref{lem:ops_ah_props}
\begin{equation*}
    \xhnorm{\mathcal{A}_h F_h} \,\sim\, \xhnorm{\mathcal{A}_h^*F_h} \,=\, 
    \xhnorm{D_{h,0}-D_{\infty,0}}.    
\end{equation*} 
\end{proof}

We finish this section with Lemma \ref{lem:ops_bh_props}, which gathers the regularity properties of the solution to the discrete Poisson problem corresponding to the fourth line in \eqref{eq:vpfp1d_hermite_semidiscrete_riesz}:
find $(\Phi_h,E_h)\in V_h\times W_h$ such that
\begin{equation} \label{eq:ops_bh_elliptic}
    E_h = -\mathcal{B}_h \Phi_h \qc -\mathcal{B}_h^* E_h = D_{h,0}-D_{\infty,0}\,.
\end{equation}
The finite element spaces $(V_h, W_h)$ are defined according to the choice of $(\mathcal{B}_h,\mathcal{B}_h^*)$ in  Section \ref{sec:disc_poisson}. 

\begin{lemma} \label{lem:ops_bh_props}
    The following properties hold true:
    \begin{enumerate}[label=(\roman*), wide]
        \item \label{enum:ops_bh_poincare_wirtinger} $H^1$ regularity. The solution $(\Phi_h, E_h)$  to \eqref{eq:ops_bh_elliptic} satisfies
        \begin{equation} \label{eq:ops_bh_poincare_wirtinger}
             \xhnorm{E_h}\lesssim\xhnorm{D_{h,0}-D_{\infty,0}} \,,
        \end{equation}
        \item \label{enum:ops_bh_electric_regularity_estimate} discrete elliptic regularity in $L^\infty$. The solution $(\Phi_h, E_h)$  to \eqref{eq:ops_bh_elliptic} satisfies
        \begin{equation} \label{eq:ops_bh_electric_regularity_estimate}
            \xhnorm[\infty]{E_h} \lesssim \xhnorm{D_{h,0}-D_{\infty,0}}.
        \end{equation}
    \end{enumerate}
    The implicit constants in the estimates above  depend only on the temperature $T_0$, the domain length $\abs{\mbT}$, the polynomial degree $m$, and the mesh quasi-uniformity constant $C_{qu}$. 
\end{lemma}
\begin{proof}
    We prove item \ref{enum:ops_bh_poincare_wirtinger} for the Raviart-Thomas scheme only, since the discontinuous Galerkin case follows the exact same lines as the proof of item \ref{enum:discrete_H1_regularity_estimate} in Lemma \ref{lem:ops_ah_discrete_regularity_estimate}. In the Raviart-Thomas case, there exists $\omega \in W_h = U_h^{m+1}\cap C^0(\mbT)$ such that  
    \begin{equation*}\label{marker:QX7}
        \partial_h \omega = \Phi_h \qand \xhdual{\omega}{1} = 0,
      \end{equation*}
      {where the broken derivative $\partial_h$ was defined in \eqref{eq:broken_derivative},}
     and therefore, we have
     \[
      \xhnorm{\Phi_h}^2 = \xhdual{\Phi_h}{\partial_h \omega}\,=\, -\,b_h(\Phi_h, \omega)
      \,,
     \]
     where $b_h$ is defined in \eqref{eq:ops_bh_rt}. Next, we use \eqref{eq:vpfp1d_hermite_semidiscrete_poisson} with $w= \omega$ and apply the Cauchy-Schwarz  inequality, leading to
      \[
     \xhnorm{\Phi_h}^2 \,=\,
     \xhdual{E_h}{\omega}
     \leq \xhnorm{E_h}\xhnorm{\omega}
     \,.
     \]
      Therefore, applying the Poincar{\'e}-Wirtinger inequality on $\omega\in H^{1}(\mathbb{T})$ with $\xhdual{\omega}{1} = 0$, we obtain  $\|\omega\|_{L^2}\lesssim \|\partial_x \omega\|_{L^2}\,=\, \|\Phi_h\|_{L^2}$, hence we derive
    \begin{equation}\label{Poinc:Poiss}
 \xhnorm{\Phi_h}   \,\lesssim\, \xhnorm{E_h}.
    \end{equation}
  On the other hand, taking $w= E_h$ in \eqref{eq:vpfp1d_hermite_semidiscrete_poisson}, using the duality property \eqref{eq:ops_bh_discrete_duality} and the second equation in \eqref{eq:vpfp1d_hermite_semidiscrete_poisson},  we have 
 \begin{eqnarray*}
   \xhnorm{E_h}^2 &=&  -b_h (\Phi_h,E_h) = -b^*_h (E_h,\Phi_h) = \xhdual{\Phi_h}{D_{h,0}-D_{\infty,0}} \,.
    \end{eqnarray*}
    We apply the Cauchy-Schwarz inequality in the right hand side and \eqref{Poinc:Poiss}, which yields item \ref{enum:ops_bh_poincare_wirtinger}
\begin{eqnarray*}
	\xhnorm{E_h}
	&\lesssim&\xhnorm{D_{h,0}-D_{\infty,0}}.
\end{eqnarray*}

    We then prove item~\ref{enum:ops_bh_electric_regularity_estimate}. The discrete Sobolev embedding theorem (cf. Theorem~5.3 in~\cite{dipietro_mathematical_2012}) states that, in the one-dimensional case, the $L^\infty$ norm of a piecewise polynomial can be controlled by its discontinuous Galerkin norm. This means
    \begin{equation*}
        \xhnorm[\infty]{E_h} \lesssim \xdhseminorm{E_h},
    \end{equation*}
where the seminorm is enough since $E_h$ has zero mean.
      Indeed
      $$
      \xhdual{E_h}{1}=-\xbhbimap{\Phi_h}{1}=0
      $$ for both choices of $b_h$ noting that $1\in W_h$.
    Therefore, it suffices to show that 
    \begin{equation*}
        \xdhseminorm{E_h} \lesssim \xhnorm{D_{h,0}-D_{\infty,0}}.
    \end{equation*}
    In the local discontinuous Galerkin case, $\mathcal{B}_h^*\,= \,\mathcal{A}_h^*\,/\,\sqrt{T_0}$ or $\mathcal{B}_h^*\,= -\,\mathcal{A}_h\,/\,\sqrt{T_0}$ and therefore we apply \eqref{eq:ops_ah_discrete_primal_dual_balance} in Lemma \ref{lem:ops_ah_props} and substitute $\mathcal{B}_h^* E_h$ thanks to \eqref{eq:ops_bh_elliptic}
    \begin{equation*}
    	\xdhseminorm{E_h} \sim 
        \xhnorm[2]{\mathcal{B}_h^* E_h}\,=\,
    	 \xhnorm{D_{h,0}-D_{\infty,0}}.
    \end{equation*}
    As for the Raviart-Thomas case, since $E_h \in C^0(\mbT)$, we have 
    \[\xdhseminorm{E_h} = \xhnorm{\partial_hE_h}\,=\,\xhnorm{D_{h,0}-D_{\infty,0}}\,. \] 
\end{proof}


\subsection{Proof of Theorem \ref{thm:vpfp1d_hermite_semidiscrete_main_theorem}}
\label{sec:proof_of_main_theorem}

Throughout this proof, we consider the solution 
$(D_h,E_h)$ to
\eqref{eq:vpfp1d_hermite_semidiscrete_riesz}-\eqref{eq:ops_bh_riesz}
as well as its associated modified relative entropy $\mathcal{H}_h(t)$
given in \eqref{eq:H0} and linearized energy functional
$\mathcal{E}_h(t)$ given by \eqref{Eh} and defined for all $t\geq
0$. We recall that the initial data satisfies the compatibility
  condition
  $$\xhdual{D^{\textrm{in}}_{h,0}-D_{\infty,0}}{1}=0,
  $$
  required to solve the discrete Poisson equation. By the conservation of mass (Proposition \ref{prop1:disc:inv} or \ref{prop2:disc:inv}), it holds $\xhdual{D_{h,0}(t)-D_{\infty,0}}{1}=0$ for all $t\geq0$, so that Lemmas \ref{lem:ops_ah_discrete_regularity_estimate} and \ref{lem:ops_bh_props} apply at any time. All the estimates below hold on the existence interval of the solution, which is $\mathbb{R}^+$ (see the comments after Theorem \ref{thm:vpfp1d_hermite_semidiscrete_main_theorem}). The proof proceeds in two main steps. The first step concerns the equivalence between the modified relative entropy $\mathcal{H}_h(t)$ and the energy functional $\mathcal{E}_h(t)$ associated to the linearized system \eqref{eq:vpfp1d_hermite_semidiscrete_riesz:linz}.
\begin{lemma} \label{lem:vpfp1d_hermite_semidiscrete_free_energy_equivalence}
Let $(D_h,E_h)$ be the solution to \eqref{eq:vpfp1d_hermite_semidiscrete_riesz}-\eqref{eq:ops_bh_riesz} and consider for all $t\geq 0$ the associated modified relative entropy $\mathcal{H}_h(t)$ given in \eqref{eq:H0} and energy functional $\mathcal{E}_h(t)$ associated to the linearized system \eqref{eq:vpfp1d_hermite_semidiscrete_riesz:linz} given by \eqref{Eh}.
    There exists a positive constant $\bar{\alpha}_0$ such that, for all $\alpha_0\in (0, \bar{\alpha}_0]$ and all $t\in \mathbb{R}^+$ 
    \begin{equation} \label{eq:vpfp1d_hermite_semidiscrete_free_energy_equivalence}
        \frac{1}{2}\,\mathcal{E}_h(t)\le \mathcal{H}_h(t) \le \frac{3}{2}\,\mathcal{E}_h(t).
    \end{equation}
    Furthermore, we have $\bar{\alpha}_0\,=\, 1/C$ for some $C>0$ {depending only on $T_0$, $|\mathbb{T}|$, $m$ and $C_{qu}$}. 
\end{lemma}
\begin{proof}
We estimate the additional term in the definition of $\mathcal{H}_h$ thanks to the Cauchy-Schwarz inequality and we apply \eqref{eq:ops_ah_discrete_H1_regularity_estimate} in Lemma \ref{lem:ops_ah_discrete_regularity_estimate} to bound $\xhnorm{F_h}$, which gives
\begin{equation*}
    |\xhdual{D_{h,1}}{F_h} |\,\le\, \xhnorm{D_{h,1}}\,\xhnorm{F_h}
    \,\le\, C\,\xhnorm{D_{h,1}}\,\xhnorm{D_{h,0}-D_{\infty,0}} \,\le\, C\,D_{\infty,0}\,\mathcal{E}_h\,,
\end{equation*}
for some positive constant $C$ depending only {on
  $T_0$, $|\mathbb{T}|$, $m$ and $C_{qu}$}, where we used
  $$
  \xhnorm{D_{h,1}}\xhnorm{D_{h,0}-D_{\infty,0}}\le\frac12\big(\xhnorm{D_{h,1}}^2+\xhnorm{D_{h,0}-D_{\infty,0}}^2\big)\le
  D_{\infty,0}\,\mathcal{E}_h
  $$
  by \eqref{Eh}. It follows from the definition \eqref{eq:H0} that
\begin{equation*}
    (1 -\alpha_0\,C)\,\mathcal{E}_h\,\le\, \mathcal{H}_h\,\le\, (1 + \alpha_0 \,C)\,\mathcal{E}_h\,.
\end{equation*}
We obtain \eqref{eq:vpfp1d_hermite_semidiscrete_free_energy_equivalence} by taking $\alpha_0 \in (0,\bar{\alpha}_0]$ with $\bar{\alpha}_0 = 1/(2C)$.
\end{proof}

The second step consists in deriving a differential inequality for  $\mathcal{H}_h$ where, unlike in the linearized energy estimate of Proposition \ref{prop:4:2}, the coercivity is regained thanks to the additional term in the definition \eqref{eq:H0} of $\mathcal{H}_h$, and where the nonlinear contributions due to the coupling with the Poisson equation are controlled by $\mathcal{H}_h^2$.
\begin{lemma} \label{lem:vpfp1d_hermite_semidiscrete_free_energy_generalized_gronwall_condition}
Let $(D_h,E_h)$ be the solution to \eqref{eq:vpfp1d_hermite_semidiscrete_riesz}-\eqref{eq:ops_bh_riesz} and consider for all $t\geq 0$ the associated modified relative entropy $\mathcal{H}_h(t)$ given in \eqref{eq:H0}.
    There exists a positive constant $C$, with $C\geq\max(2,1/\bar{\alpha}_0)$ where $\bar{\alpha}_0$ is given by Lemma \ref{lem:vpfp1d_hermite_semidiscrete_free_energy_equivalence}, such that, for all $\alpha_0\in (0, \min(\tau_0,\tau_0^{-1})/C]$ and all $t\in \mathbb{R}^+$, it holds
    \begin{equation*}
        \frac{\dd{\mathcal{H}_h}}{\dd{t}}(t) \,+\, \frac{2\alpha_0}{3}\,\mathcal{H}_h(t)
        \,\le\, C\max(\tau_0,\tau_0^{-1})\,
        \mathcal{H}_h^2(t)
        \,.
\end{equation*}
In the previous estimate, the constant $C$ only depends on $\rho_0$, $T_0$, $|\mathbb{T}|$, $m$ and $C_{qu}$.
\end{lemma}
\begin{proof}
For simplicity, we omit the dependence on $t$ in this proof. To compute the time derivative of $\mathcal{H}_h$, we decompose it as follows 
\begin{equation}\label{dt:H:h}
 \frac{\dd\mathcal{H}_h}{\dd{t}}
 \,=\,
  \frac{\dd\mathcal{E}_h}{\dd{t}} 
  \,-\,
  \frac{\alpha_0}{D_{\infty,0}}\,\left(
  \xdual{D_{h,1}}{\partial_t F_h}
  \,+\,
  \xdual{\partial_t D_{h,1}}{F_h}\right).
\end{equation}
On the one hand, to compute the time derivative of the linearized energy functional $\mathcal{E}_h$, we follow the same lines as in the proof of Proposition \ref{prop:4:2}: we sum the $L^2$-inner products of the first, second and third lines in \eqref{eq:vpfp1d_hermite_semidiscrete_riesz} with $ D_{h,0} - D_{\infty,0}$, $ D_{h,1}$ and $ D_{h,k}$, $k\in \{2,\dots, N_H\}$ respectively, divided by $D_{\infty,0}$, leading to
\[
\frac{\dd\mathcal{E}_h}{\dd{t}}\,+\,  \frac{1}{\tau_0}\mathcal{I}_h
\,=\, \mathcal{R}_h\,,
\]
where  $\mathcal{R}_h$ is the nonlinear residual
$$
\mathcal{R}_h \,\defeq\, \frac{1}{D_{\infty,0}}\sum_{k=0}^{N_H-1} \sqrt{\frac{k+1}{T_0}}\;\xhdual{\Pi_h(E_h(D_{h,k}-D_{\infty,k}))}{D_{h,k+1}}\,,
$$
whereas $\mathcal{I}_h$ is the entropy dissipation given by \eqref{def:I:h}. Since $D_{h,k+1}\in U^m_h$, we may rewrite $\mathcal{R}_h$ without the projection $\Pi_h$, that is,
\[
\mathcal{R}_h \,=\, \frac{1}{D_{\infty,0}}\sum_{k=0}^{N_H-1} \sqrt{\frac{k+1}{T_0}}\,\xhdual{E_h(D_{h,k}-D_{\infty,k})}{D_{h,k+1}}\,.
\]
On the other hand, to compute the third term in \eqref{dt:H:h}, we consider the $L^2$-inner product of $ F_{h}$ with the second line in \eqref{eq:vpfp1d_hermite_semidiscrete_riesz}. Hence, we obtain
$$
\frac{\dd\mathcal{H}_h}{\dd{t}} \;+\; \frac{1}{\tau_0}\,\mathcal{I}_h \;+\; \frac{\alpha_0}{D_{\infty,0}}\,\widetilde{\mathcal{I}}_h \;=\; \mathcal{R}_h \,+\; \frac{\alpha_0}{D_{\infty,0}}\;\widetilde{\mathcal{R}}_{h}, 
$$
with $\widetilde{\mathcal{I}}_h\,=\,\widetilde{\mathcal{I}}_{h,1}\,+\,\widetilde{\mathcal{I}}_{h,2}\,+\,\widetilde{\mathcal{I}}_{h,3}$, where
\begin{equation*}
    \left\{
    \begin{array}{l}
        \ds \widetilde{\mathcal{I}}_{h,1} \;\defeq\; -\,\xhdual{\mathcal{A}_h D_{h,0}}{F_h} + \sqrt{2}\xhdual{\mathcal{A}_h^*D_{h,2}}{F_h} - \frac{1}{\tau_0}\xhdual{D_{h,1}}{F_h}, 
        \\ [1.em]
        \ds \widetilde{\mathcal{I}}_{h,2} \;\defeq \,-
        \frac{D_{\infty,0}}{T_0}\,\xhdual{\mathcal{A}_h\Phi_h}{F_h}\,,
        \\ [1.em]
        \ds \widetilde{\mathcal{I}}_{h,3} \,\defeq\; \xhdual{D_{h,1}}{\partial_t F_h}
    \end{array}
    \right.
\end{equation*}
and  $\widetilde{\mathcal{R}}_{h}$ is given by
$$
\widetilde{\mathcal{R}}_{h} \;\defeq\; -\,\frac{1}{\sqrt{T_0}}\xdual{\Pi_h(E_h(D_{h,0}-D_{\infty,0}))}{F_h}.
$$
Here, the contributions $\widetilde{\mathcal{I}}_{h}$ and $\widetilde{\mathcal{R}}_{h}$ come from the additional term in the definition of the {modified relative entropy} $\mathcal{H}_h$. In what follows, we denote by $C > 0$ a constant that depends only on $\rho_0$, $T_0$, $\abs{\mbT}$, the polynomial degree $m$, and the mesh quasi-uniformity constant $C_{qu}$. 

To lower bound $\widetilde{\mathcal{I}}_{h,1}$,
we rewrite its first term using \eqref{eq:ops_ah_discrete_kernel}, applying the duality relation \eqref{eq:ops_ah_discrete_duality}, and substituting $\mathcal{A}_h^*F_h$ according to \eqref{eq:ops_ah_elliptic}. This yields
\begin{equation*}
\begin{aligned}
    -\xhdual{\mathcal{A}_h D_{h,0}}{F_h} \,=\, -\xhdual{\mathcal{A}_h (D_{h,0}-D_{\infty,0})}{F_h} &\;=\, -\xhdual{D_{h,0}-D_{\infty,0}}{\mathcal{A}_h^*F_h} \\
    &\,=\; \xhnorm{D_{h,0}-D_{\infty,0}}^2 \,.
 \end{aligned}
\end{equation*}
Next, we obtain a lower bound on the second term in the definition of $\widetilde{\mathcal{I}}_{h,1}$ thanks to the duality relation \eqref{eq:ops_ah_discrete_duality}, the Young inequality, and the regularity estimate \eqref{eq:ops_ah_discrete_H2_regularity_estimate} in Lemma \ref{lem:ops_ah_discrete_regularity_estimate}, which yields for all $\eta >0$,
\begin{equation*}
    \sqrt{2}\,\xhdual{\mathcal{A}_h^*D_{h,2}}{F_h} = \sqrt{2}\,\xhdual{D_{h,2}}{\mathcal{A}_h F_h} \,\geq\, -\,\frac{D_{\infty,0}}{\eta} \,\mathcal{I}_h\,-\,
    C\eta\xhnorm{D_{h,0}-D_{\infty,0}}^2.
\end{equation*}
 Similarly, we estimate from below the third term of  $\widetilde{\mathcal{I}}_{h,1}$ thanks to the Young inequality and the regularity estimate \eqref{eq:ops_ah_discrete_H1_regularity_estimate} in Lemma \ref{lem:ops_ah_discrete_regularity_estimate},
\begin{equation*}
    -\frac{1}{\tau_0}\;\xhdual{D_{h,1}}{F_h} \,\geq\; -\frac{1}{\tau_0}\;\xhnorm{D_{h,1}}\xhnorm{F_h} \,\geq\,
    -\,\frac{D_{\infty,0}}{\eta\,\tau_0}\,\mathcal{I}_h
    \,-\,\frac{C\eta}{\tau_0}\;\xhnorm{D_{h,0}-D_{\infty,0}}^2\,.
\end{equation*}
Gathering these estimates, we find 
\begin{equation}
\label{est:1}
    \widetilde{\mathcal{I}}_{h,1} \ge \, \left(1  - {C\eta}\,(1+\tau_0^{-1}) \right)\xhnorm{D_{h,0}-D_{\infty,0}}^2
    \,-\, \frac{\tau_0^{-1}+1}{\eta}\,D_{\infty,0}\,\mathcal{I}_h\,.
\end{equation}

To evaluate $\widetilde{\mathcal{I}}_{h,2}$, we apply the duality
relation \eqref{eq:ops_ah_discrete_duality},  substitute
$\mathcal{A}_h^* F_h$ according to the second equation in
\eqref{eq:ops_ah_elliptic} and substitute $D_{h,0}-D_{\infty,0}$ according to the second equation in the fourth line of \eqref{eq:vpfp1d_hermite_semidiscrete_riesz}, that is
\begin{equation*}
    \widetilde{\mathcal{I}}_{h,2} \,=\, -\frac{D_{\infty,0}}{T_0} \xhdual{\Phi_h}{\mathcal{A}_h^*F_h} \,=\,
    \frac{D_{\infty,0}}{T_0}\xhdual{\Phi_h}{D_{h,0}-D_{\infty,0}} = -\frac{D_{\infty,0}}{T_0}\xhdual{\Phi_h}{\mathcal{B}_h^* E_h} \,. 
\end{equation*}
Next, we apply the duality relation \eqref{eq:ops_bh_discrete_duality} and substitute $\mathcal{B}_h\Phi_h$ according to the first equation in the fourth line of \eqref{eq:vpfp1d_hermite_semidiscrete_riesz}, leading to
\begin{equation}
    \widetilde{\mathcal{I}}_{h,2} = 
    -\frac{D_{\infty,0}}{T_0}\xhdual{\mathcal{B}_h\Phi_h}{E_h}  
    = \frac{D_{\infty,0}}{T_0}\xhnorm{E_h}^2\,\geq \,0\,.
\label{est:2}
\end{equation}

Then we treat the last term $\widetilde{\mathcal{I}}_{h,3}$ and first evaluate $\xhnorm{\partial_t F_h}^2$. We substitute $F_h$ according to the first equation in \eqref{eq:ops_ah_elliptic} and we apply the duality relation \eqref{eq:ops_ah_discrete_duality}, leading to 
\begin{equation*}
    \xhnorm{\partial_t F_h}^2 \,=\,-\, \xhdual{\partial_t F_h}{\partial_t \mathcal{A}_h\Psi_h}
    \,=\,-\, \xhdual{\partial_t\mathcal{A}_h^* F_h}{\partial_t\Psi_h} \,.
\end{equation*}
{To compute the right hand side, we first take the time
  derivative of the second equation in \eqref{eq:ops_ah_elliptic} and
  then test the resulting relation with $\partial_t\Psi_h$, which yields
\begin{equation*}\label{marker:Q1Xb}
	\xhnorm{\partial_t F_h}^2 \,=\,\xhdual{\partial_t(D_{h,0}-D_{\infty,0})}{\partial_t\Psi_h} \,.
\end{equation*}}
Next, we substitute $\partial_t(D_{h,0}-D_{\infty,0})$ according to the first equation in \eqref{eq:vpfp1d_hermite_semidiscrete_riesz}, apply the duality relation \eqref{eq:ops_ah_discrete_duality}, substitute $\mathcal{A}_h \Psi_h$ according to the first equation in \eqref{eq:ops_ah_elliptic}
 and we apply the Cauchy-Schwarz inequality to get that
 \begin{equation*}
    \xhnorm{\partial_t F_h}^2 \,=\,
    \xhdual{\mathcal{A}_h^*D_{h,1}}{\partial_t\Psi_h}
    \,=\,
    \xhdual{D_{h,1}}{\partial_t \mathcal{A}_h\Psi_h}
    \,\leq \,
    \xhnorm{D_{h,1}}
    \xhnorm{\partial_t F_{h}}\,.    
\end{equation*}
Hence, to compute a lower bound on $\widetilde{\mathcal{I}}_{h,3}$, we apply the Cauchy-Schwarz inequality and bound $\xhnorm{\partial_t F_{h}}$ thanks to the previous estimate 
\begin{equation}
\label{est:3}
    \widetilde{\mathcal{I}}_{h,3} \,\ge\, -\xhnorm{D_{h,1}}^2
    \,\ge\, -\,D_{\infty,0}\,\mathcal{I}_h\,.
\end{equation}

We then estimate the nonlinear residual $\mathcal{R}_h$ using Young's inequality 
\begin{eqnarray*}
    \mathcal{R}_h &=& \frac{1}{D_{\infty,0}}\sum_{k=0}^{N_H-1}\sqrt{\frac{k+1}{T_0}}\xhdual{E_h(D_{h,k}-D_{\infty,k})}{D_{h,k+1}} \\[0.6em]
     &\le & \frac{1}{2\,\tau_0 }\,\mathcal{I}_h \,+\, \frac{\tau_0}{2\,D_{\infty,0}\,T_0}\sum_{k=0}^{N_H-1}\xhnorm[\infty]{E_h}^2\xhnorm{D_{h,k}-D_{\infty,k}}^2\,.
\end{eqnarray*}
Applying \eqref{eq:ops_bh_electric_regularity_estimate} in Lemma \ref{lem:ops_bh_props} to estimate $\xhnorm[\infty]{E_h}$, we obtain
\begin{equation}
\label{est:4}
    \mathcal{R}_h \,\le\; \frac{1}{2\,\tau_0}\,\mathcal{I}_h \,+\,
    C\tau_0\,\mathcal{E}_h^2.
\end{equation}

To control the nonlinear term $\widetilde{\mathcal{R}}_{h}$, we first remove $\Pi_h$ using that $F_h\in U^m_h$ and then, we apply the Young inequality and bound $E_h$ by its supremum over $\mathbb{T}$,  which yields for all $\eta>0$,
\begin{equation*}
    \widetilde{\mathcal{R}}_{h}
    \,\le\, \frac{1}{2\,\eta}\,\xhnorm[\infty]{E_h}^2\xhnorm{D_{h,0}-D_{\infty,0}}^2\,+\,
    \frac{\eta}{2\,T_0}\,\xhnorm{F_h}^2
    \,.
\end{equation*}
Next, we again apply \eqref{eq:ops_bh_electric_regularity_estimate} in Lemma \ref{lem:ops_bh_props} to bound $\|E_h\|_{L^\infty(\mathbb{T})}$ and \eqref{eq:ops_ah_discrete_H1_regularity_estimate} in Lemma \ref{lem:ops_ah_discrete_regularity_estimate} to bound $\|F_h\|_{L^2(\mathbb{T})}$, leading to
\begin{equation}
\label{est:5}
    \widetilde{\mathcal{R}}_{h}
    \,\le\, C \,\left(\frac{1}{\eta}\,\mathcal{E}_h^2\,+\, 
    \eta\xhnorm{D_{h,0}-D_{\infty,0}}^2\right)\,.
\end{equation}

Finally, gathering the above estimates \eqref{est:1}-\eqref{est:5}, and absorbing the term $C\eta\,\xhnorm{D_{h,0}-D_{\infty,0}}^2$ of \eqref{est:5} in the one of \eqref{est:1} up to a change of $C$, we  derive the following differential inequality
\begin{equation*}
    \begin{aligned}
        \frac{\dd{\mathcal{H}_h}}{\dd{t}} &\,+\, 
        \frac{\alpha_0}{D_{\infty,0}}\left(1-
        C\eta\,(1+\tau_0^{-1})
        \right)\,
        \xhnorm{D_{h,0}-D_{\infty,0}}^2
        \,+\,
        \frac{\alpha_0}{T_0}\xhnorm{E_h}^2\\[0.5em]
        &+\frac{1}{\tau_0}\,\left(
        \frac{1}{2}
        \;-\;
        \alpha_0
        \left(
        \tau_0+\frac{1+\tau_0}{\eta}
        \right)\right)\,\mathcal{I}_h \,\le\, C\,\left(\tau_0+\frac{\alpha_0}{\eta}\right)\,\mathcal{E}_h^2,
    \end{aligned}
\end{equation*}
which is valid for all $\eta>0$.  Thus we fix $\eta\,=\, 1/(2\;C\;(1+\tau^{-1}_0))$ in order to get the dissipation in $\xhnorm{D_{h,0}-D_{\infty,0}}^2$, which yields
\begin{equation*}
\begin{aligned}
        \frac{\dd{\mathcal{H}_h}}{\dd{t}} \,+\, 
        \frac{\alpha_0}{2 D_{\infty,0}}
        \xhnorm{D_{h,0}-D_{\infty,0}}^2
        &\,+\,
        \frac{\alpha_0}{T_0}\xhnorm{E_h}^2\,+\,\frac{1}{\tau_0}\left(
        \frac{1}{2}
        \,-\,
        C\,\alpha_0\,\max(\tau_0,\tau_0^{-1})\right)\mathcal{I}_h \\[0.8em]
        &\,\le\, C\;(1+\alpha_0)\;\max(\tau_0,\tau_0^{-1})\,\mathcal{E}_h^2\,,
\end{aligned}
\end{equation*}
where we used $(1+\tau_0)^2/\tau_0\le 4\max(\tau_0,\tau_0^{-1})$.
Then we choose the free parameter $\alpha_0>0$ in order to get the
dissipation in $\mathcal{I}_h$:  there exists a constant $C>0$, only
depending on $\rho_0$, $T_0$, $|\mathbb{T}|$, $m$
and $C_{qu}$, such that for all $\alpha_0\in(0,
\min(\tau_0,\tau_0^{-1})/(4C)]$, which ensures $C\,\alpha_0\max(\tau_0,\tau_0^{-1})\le 1/4$ and $\alpha_0\le1$, we have, up to a change of $C$,
\begin{equation*}
        \frac{\dd{\mathcal{H}_h}}{\dd{t}} \,+\, 
        \frac{\alpha_0}{2 D_{\infty,0}}
        \xhnorm{D_{h,0}-D_{\infty,0}}^2
        \,+\,
        \frac{\alpha_0}{T_0}\xhnorm{E_h}^2\,+\,\frac{1}{4\,\tau_0}\,\mathcal{I}_h \,\le\, C\max(\tau_0,\tau_0^{-1})\,\mathcal{E}_h^2\,.
\end{equation*}
Finally we use  that 
$$
\mathcal{I}_h \;\geq\, \frac{1}{D_{\infty,0}}\sum_{k=1}^{N_H} \xhnorm{D_{h,k}-D_{\infty,k}}^2 \,,
$$ 
and since $\alpha_0\le\min(\tau_0,\tau_0^{-1})/C\le
  1/(2\,\tau_0)$, recalling that $C\geq2$, we have $\frac{1}{4\tau_0}\geq\frac{\alpha_0}{2}$, so that, by \eqref{Eh}, it holds that
\begin{equation*}
        \frac{\dd{\mathcal{H}_h}}{\dd{t}} \,+\, \alpha_0\,\mathcal{E}_h
        \,\le\, C\max(\tau_0,\tau_0^{-1})\,
        \mathcal{E}_h^2
        \,.
\end{equation*}
We lower and upper bound $\mathcal{E}_h$ on the left and right hand
side respectively thanks to Lemma
\ref{lem:vpfp1d_hermite_semidiscrete_free_energy_equivalence},
  which applies since $\alpha_0\le 1/C\le\bar{\alpha}_0$, that is
  $\mathcal{E}_h\geq\frac23\mathcal{H}_h$ and
  $\mathcal{E}_h^2\le4\,\mathcal{H}_h^2$ up to a change of $C$, which yields the expected estimate, that is,
\begin{equation*}
        \frac{\dd{\mathcal{H}_h}}{\dd{t}} \,+\, \frac{2\alpha_0}{3}
        \,\mathcal{H}_h
        \,\le\, C\max(\tau_0,\tau_0^{-1})\,
        \mathcal{H}_h^2
        \,.
\end{equation*}

\end{proof}

To conclude the proof of Theorem \ref{thm:vpfp1d_hermite_semidiscrete_main_theorem}, we fix $\alpha_0=\min(\tau_0,\tau_0^{-1})/C$ in \eqref{eq:H0}, where $C$ is the constant of Lemma \ref{lem:vpfp1d_hermite_semidiscrete_free_energy_generalized_gronwall_condition}. Since $\max(\tau_0,\tau_0^{-1})=1/\min(\tau_0,\tau_0^{-1})$, the differential inequality of this lemma reads, for all $t\geq0$,
\begin{equation}\label{eq:final:ode}
        \frac{\dd{\mathcal{H}_h}}{\dd{t}}(t) \,\le\, -\,\frac{\min(\tau_0,\tau_0^{-1})}{3\,C}\,\mathcal{H}_h(t)\left(2-3\,C^2\max(\tau_0^2,\tau_0^{-2})\,\mathcal{H}_h(t)\right).
\end{equation}
We define $\kappa:=1/(3\,C)$ and assume that the initial data satisfies \eqref{eq:vpfp1d_hermite_semidiscrete_hypocoercive_condition}. By Lemma \ref{lem:vpfp1d_hermite_semidiscrete_free_energy_equivalence}, we obtain
\[
\mathcal{H}_h(0)\,\le\,\frac32\,\mathcal{E}_h(0)\,\le\,\frac{\min(\tau_0^2,\tau_0^{-2})}{6\,C^2}\,.
\]
According to \eqref{eq:final:ode}, $\mathcal{H}_h$ is nonincreasing as
long as $\mathcal{H}_h \leq \min(\tau_0^2,\tau_0^{-2})/(3\,C^2)$; since
$\mathcal{H}_h$ is continuous and $\mathcal{H}_h(0)$ lies strictly
below this threshold, $\mathcal{H}_h$ never reaches it. Indeed, otherwise, at the first time $t_1$ where it is reached,
  $\mathcal{H}_h$ would be nonincreasing on $[0,t_1]$, hence
  $\mathcal{H}_h(t_1)\le\mathcal{H}_h(0) < \min(\tau_0^2,\tau_0^{-2})/(3\,C^2)$, leading to a contradiction. Hence
\[
\mathcal{H}_h(t)\,\le\,\mathcal{H}_h(0)\,\le\,\frac{\min(\tau_0^2,\tau_0^{-2})}{6\,C^2}\,,\quad\forall\,t\geq0\,.
\]
The term in parentheses in \eqref{eq:final:ode} is therefore larger than $3/2\geq1$, so that
\[
\frac{\dd{\mathcal{H}_h}}{\dd{t}}\,\le\,-\,\kappa\,\min(\tau_0,\tau_0^{-1})\,\mathcal{H}_h
\qquad\text{and}\qquad
\mathcal{H}_h(t)\,\le\,\mathcal{H}_h(0)\,\exp\left({-\kappa\,\min(\tau_0,\tau_0^{-1})\,t}\right)\,,\quad\forall\,t\geq0\,.
\]
Finally, Lemma \ref{lem:vpfp1d_hermite_semidiscrete_free_energy_equivalence} gives $\mathcal{E}_h(t)\le2\,\mathcal{H}_h(t)$ and $\mathcal{H}_h(0)\le\frac32\,\mathcal{E}_h(0)$, which leads to the estimate in Theorem \ref{thm:vpfp1d_hermite_semidiscrete_main_theorem}, that is
\[
\mathcal{E}_h(t)\,\leq \,3\,\mathcal{E}_h(0)\,
\exp{\left(-\,\kappa\,
\min(\tau_0,\tau_0^{-1})
\,t\right)}\,,\quad\forall\,t\geq0\,.
\]

\section{Numerical simulations}
\label{sec:simulation}

In this section, we present several numerical experiments using the proposed scheme \eqref{eq:vpfp1d_hermite_semidiscrete_riesz}. For the time discretization, we adopt the same splitting strategy as in \cite{blaustein_structure_2024}, combined with a second-order diagonally implicit Runge-Kutta scheme. To control spurious oscillations arising from the Hermite expansion in velocity, we apply the Hou-Li filter with the $\frac{2}{3}$ dealiasing rule \cite{hou_computing_2007, di_filtered_2019}. For the spatial discretization, we employ a second-order discontinuous Galerkin scheme ($m=1$) with alternating fluxes $(\hat{g}_a(u),\hat{g}_a^*(u)) = (u^-,u^+)$ for both Vlasov-Fokker-Planck and Poisson equations.

Throughout the simulations, the background temperature is fixed to $T_0 = 1$ and the time step is set to $\Delta t = 0.1$ unless otherwise specified. The spatial domain is taken as $\mbT = [0, L]$. The number of Hermite modes $N_H$ is chosen adaptively based on the collisional regime. In particular, three representative regimes are considered: a weakly collisional regime ($\tau_0 = 10^5$), a moderately collisional regime ($\tau_0 = 10^3$), and a strongly collisional regime ($\tau_0 = 10$). 

\subsection{Order of convergence} \label{sec:convergence_test}
In this subsection, we assess the convergence order of the proposed scheme in space-time and we test its uniformity with respect to the physical parameter $\tau_0$. With a fixed $N_H$, the convergence rate is measured on successive mesh refinements using the differences between numerical solutions defined by
\begin{equation} \label{eq:relative_error}
    \varepsilon_h^p := \norm{f_h-f_{h/2}}_{L^2(f_\infty^{-1})} = \left(\sum_{k=0}^{N_H} \xnorm[p]{D_{h,k}-D_{h/2,k}}^2\right)^{\frac{1}{2}}.
\end{equation}
We choose the initial data as a perturbation of the equilibrium, namely,
\begin{equation} \label{eq:landau_damping_ic}
    f(0,x,v) = \left(1 + \delta\cos\left( \frac{2\pi x}{L}\right)\right)\mathcal{M}(v),
\end{equation}
where the length of the spatial domain is $L = 4\pi$ whereas the
amplitude of the perturbation is $\delta = 0.05$. In the limit $\tau_0
\goto +\infty$, this setting corresponds to the classical Landau
damping configuration. We perform simulations with $\tau_0  = 10^k$
with $k \in \{1, 3, 5\}$. The final time is set to $t = 1$ and the time step is $\Delta t = 0.5\,h$.  The results reported in Table~\ref{tab:convergence_landau_damping} show that the proposed scheme achieves second order convergence in the moderately and weakly collisional regimes ($\tau_0 = 10^3$ and $\tau_0 = 10^5$), whereas in the strongly collisional regime ($\tau_0 = 10$) the observed order is slightly lower, around $1.8$, but recovers second-order convergence upon further mesh refinement.

\begin{table}[htpb]
    \renewcommand\arraystretch{1.1}
    \centering
    \small
    \begin{tabular}{|c|c|cccccc|}
        \hline
        $\tau_0$ & $N_x$
        & $\varepsilon_h^1$ & Order
        & $\varepsilon_h^2$ & Order
        & $\varepsilon_h^\infty$ & Order \\
        \hline

        \multirow{3}{*}{$10^1$}
        & 128 & 2.00E-03 & - & 6.65E-04 & - & 4.08E-04 & - \\
        & 256 & 5.72E-04 & 1.80 & 1.89E-04 & 1.81 & 1.14E-04 & 1.84 \\
        & 512 & 1.61E-04 & 1.83 & 5.32E-05 & 1.83 & 3.13E-05 & 1.86 \\
        \hline

        \multirow{3}{*}{$10^3$}
        & 128 & 3.17E-03 & - & 1.10E-03 & - & 7.16E-04 & - \\
        & 256 & 8.90E-04 & 1.83 & 3.10E-04 & 1.82 & 2.03E-04 & 1.81 \\
        & 512 & 2.11E-04 & 2.08 & 7.19E-05 & 2.11 & 4.60E-05 & 2.14 \\
        \hline

        \multirow{3}{*}{$10^5$}
        & 128 & 3.22E-03 & - & 1.11E-03 & - & 7.29E-04 & - \\
        & 256 & 8.99E-04 & 1.84 & 3.13E-04 & 1.83 & 2.06E-04 & 1.82 \\
        & 512 & 2.11E-04 & 2.09 & 7.20E-05 & 2.12 & 4.61E-05 & 2.16 \\
        \hline
    \end{tabular}
    \caption{\em {Numerical values of the errors $\varepsilon_h^1$, $\varepsilon_h^2$ and $\varepsilon_h^\infty$} and of the order of numerical convergence at time $t = 1$, starting from the initial data \eqref{eq:landau_damping_ic} with $\delta = 0.05$, $L = 4\pi$, $\Delta t = 0.5\,h$, $N_H=512$, $\tau_0 \in \set{10^1,10^3,10^5}$.}
    \label{tab:convergence_landau_damping}
\end{table}

\subsection{Comparison across collisional regimes}
\label{sec:landau_compare_behavior}
In this subsection, we investigate the behavior of the proposed scheme across different collisional regimes. The phase space is discretized using $N_x = 128$ cells in space and $N_H = 640$ Hermite modes in velocity. The initial condition is prescribed by \eqref{eq:landau_damping_ic} with perturbation amplitude $\delta = 0.5$, which corresponds to the strong Landau damping regime in the collisionless limit $\tau_0 \to +\infty$.

Figure~\ref{fig:landau_compare_distribution} displays snapshots of the distribution function $f$ at times $t \in \{4,16,40\}$. In the weakly collisional regime, the solution exhibits pronounced phase mixing, and fine-scale filamentation in phase space persists for large time. As $\tau_0$ decreases, collisional effects progressively damp high-frequency velocity modes, leading to a faster relaxation toward equilibrium. In the strongly collisional regime, phase mixing is rapidly suppressed, and the solution relaxes exponentially toward a stationary state.

\begin{figure}[!htb]
    \centering
    \begin{subfigure}[t]{0.32\textwidth}
        \centering
        \includegraphics[width=\linewidth, height=4cm, keepaspectratio]{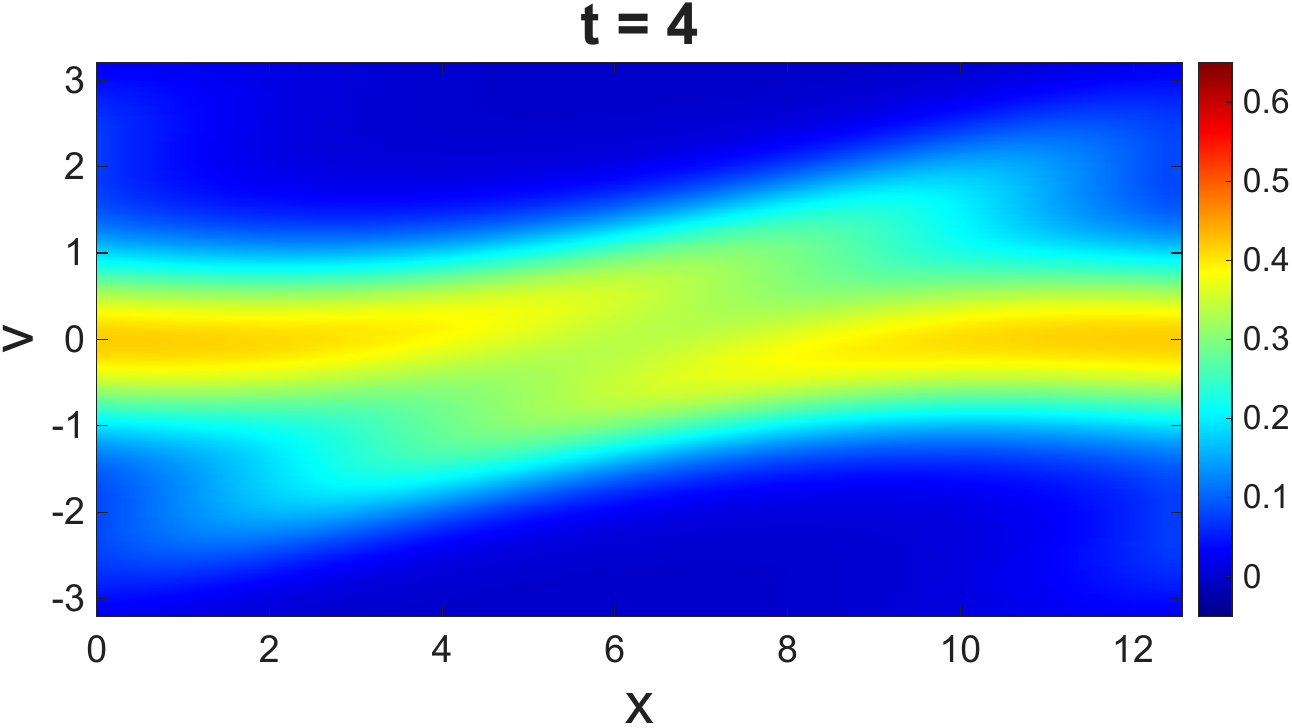}
    \end{subfigure}
    \begin{subfigure}[t]{0.32\textwidth}
        \centering
        \includegraphics[width=\linewidth, height=4cm, keepaspectratio]{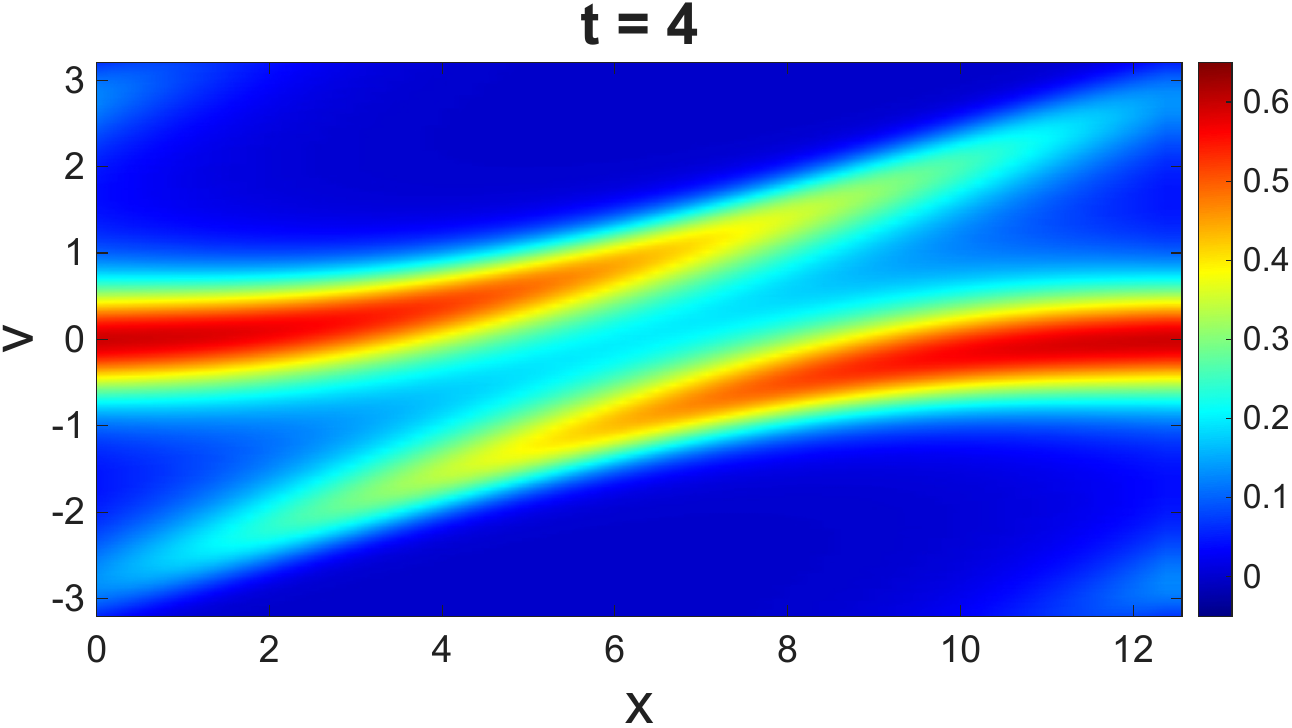}
    \end{subfigure}
    \begin{subfigure}[t]{0.32\textwidth}
        \centering
        \includegraphics[width=\linewidth, height=4cm, keepaspectratio]{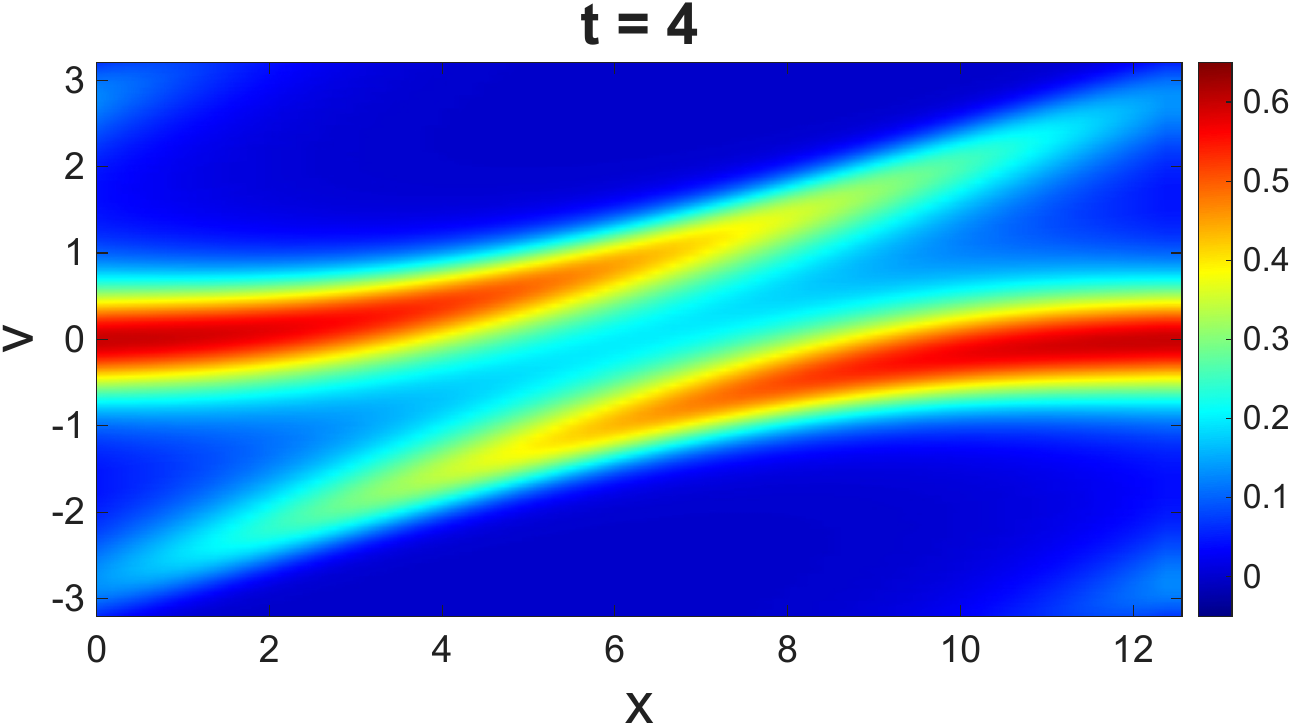}
    \end{subfigure}

    \medskip

    \begin{subfigure}[t]{0.32\textwidth}
        \centering
        \includegraphics[width=\linewidth, height=4cm, keepaspectratio]{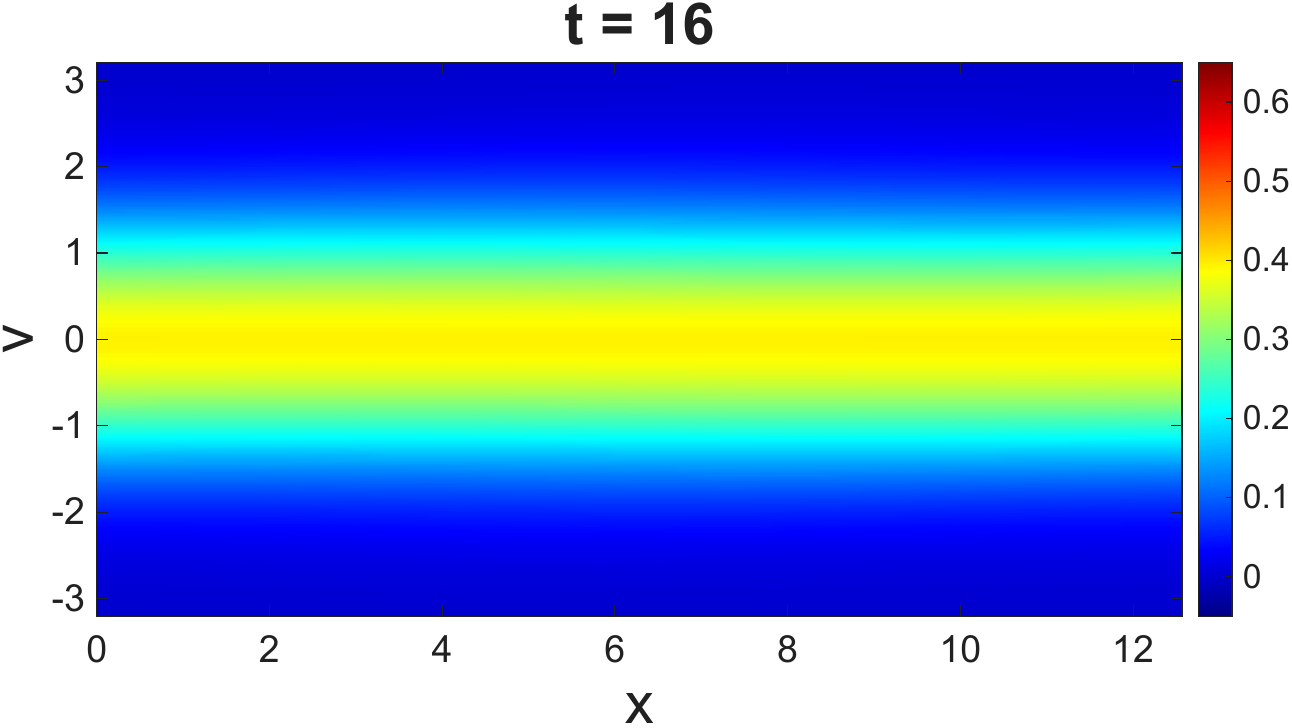}
    \end{subfigure}
    \begin{subfigure}[t]{0.32\textwidth}
        \centering
        \includegraphics[width=\linewidth, height=4cm, keepaspectratio]{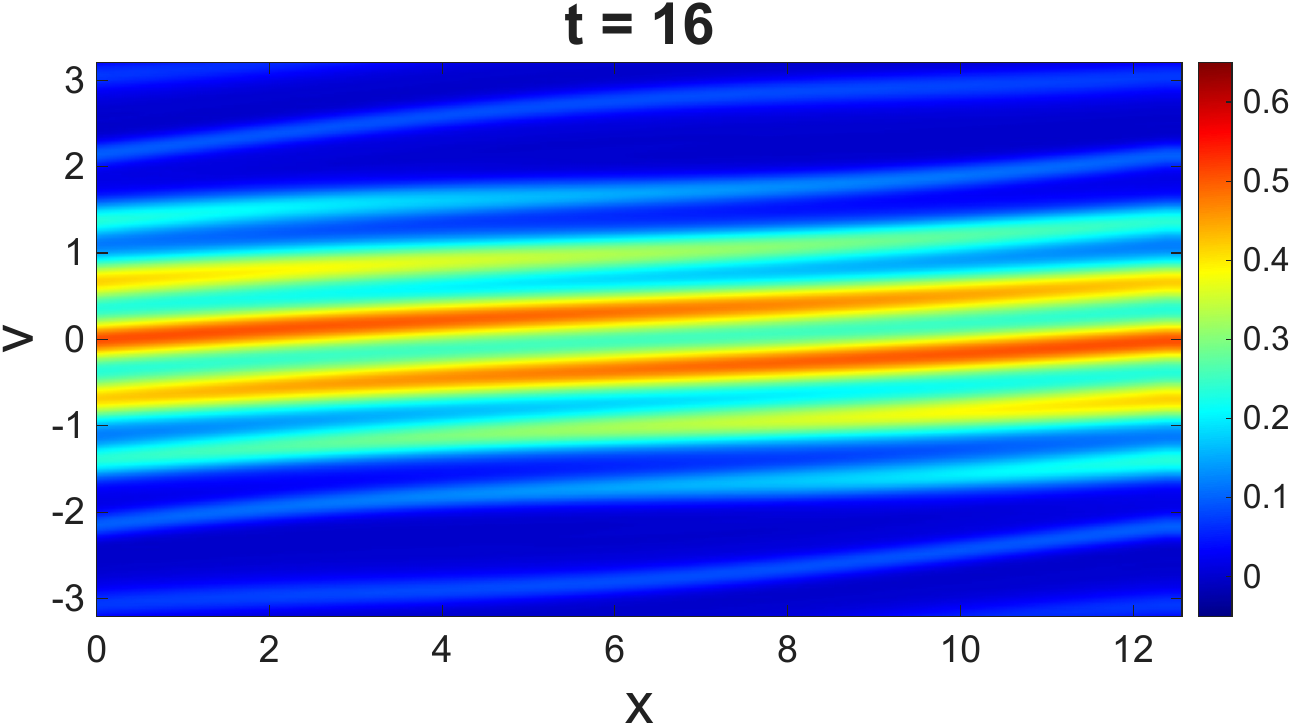}
    \end{subfigure}
    \begin{subfigure}[t]{0.32\textwidth}
        \centering
        \includegraphics[width=\linewidth, height=4cm, keepaspectratio]{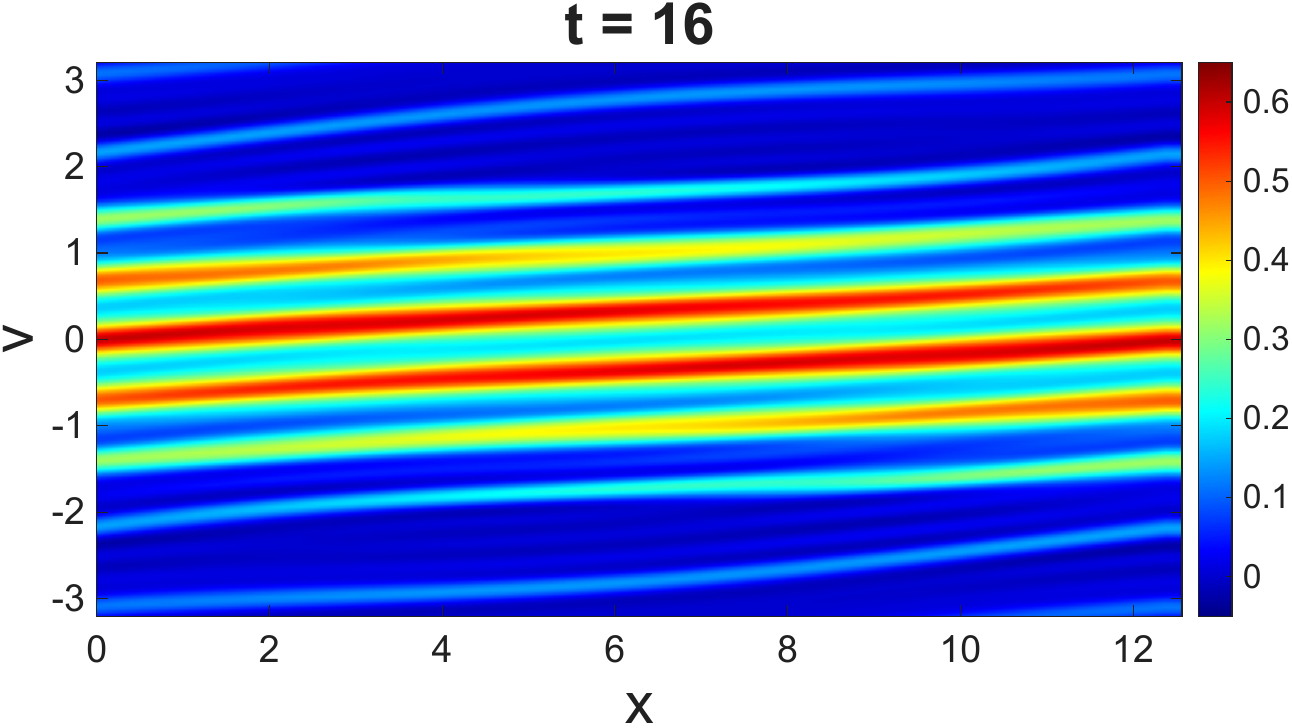}
    \end{subfigure}

    \medskip

    \begin{subfigure}[t]{0.32\textwidth}
        \centering
        \includegraphics[width=\linewidth, height=4cm, keepaspectratio]{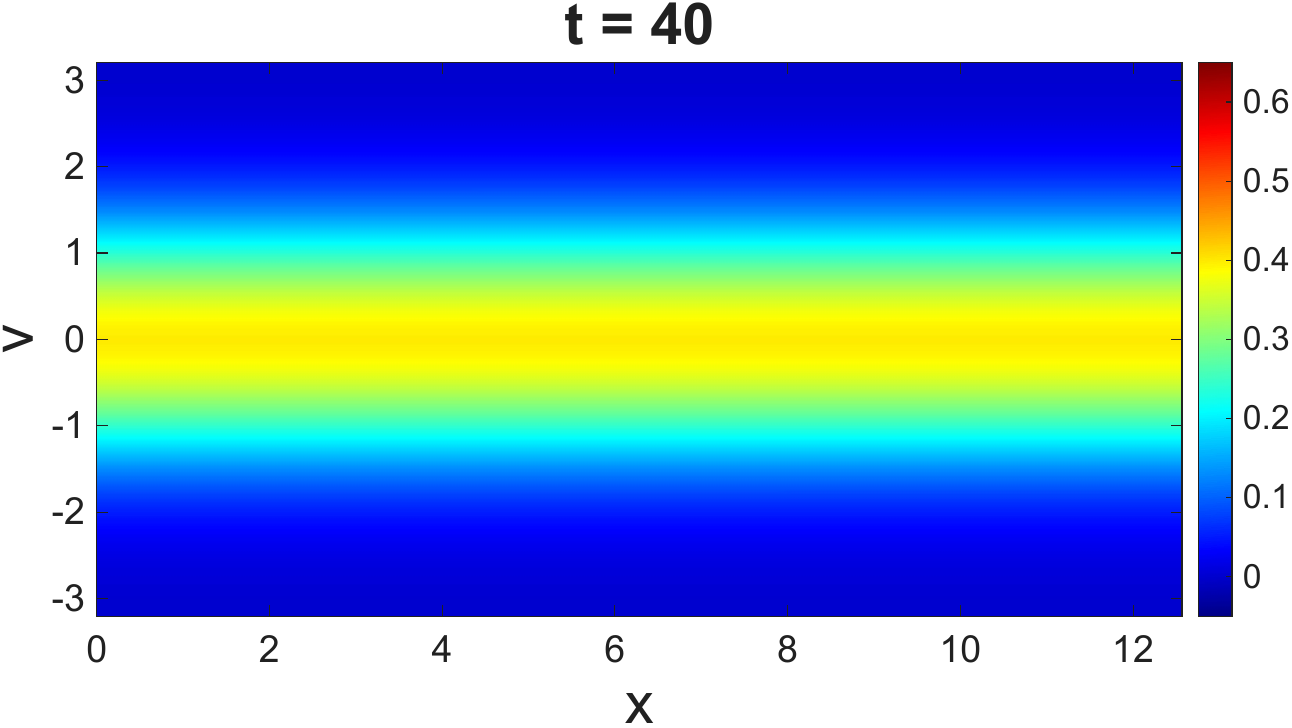}
        \subcaption{$\tau_0 = 10^1$.}
    \end{subfigure}
    \begin{subfigure}[t]{0.32\textwidth}
        \centering
        \includegraphics[width=\linewidth, height=4cm, keepaspectratio]{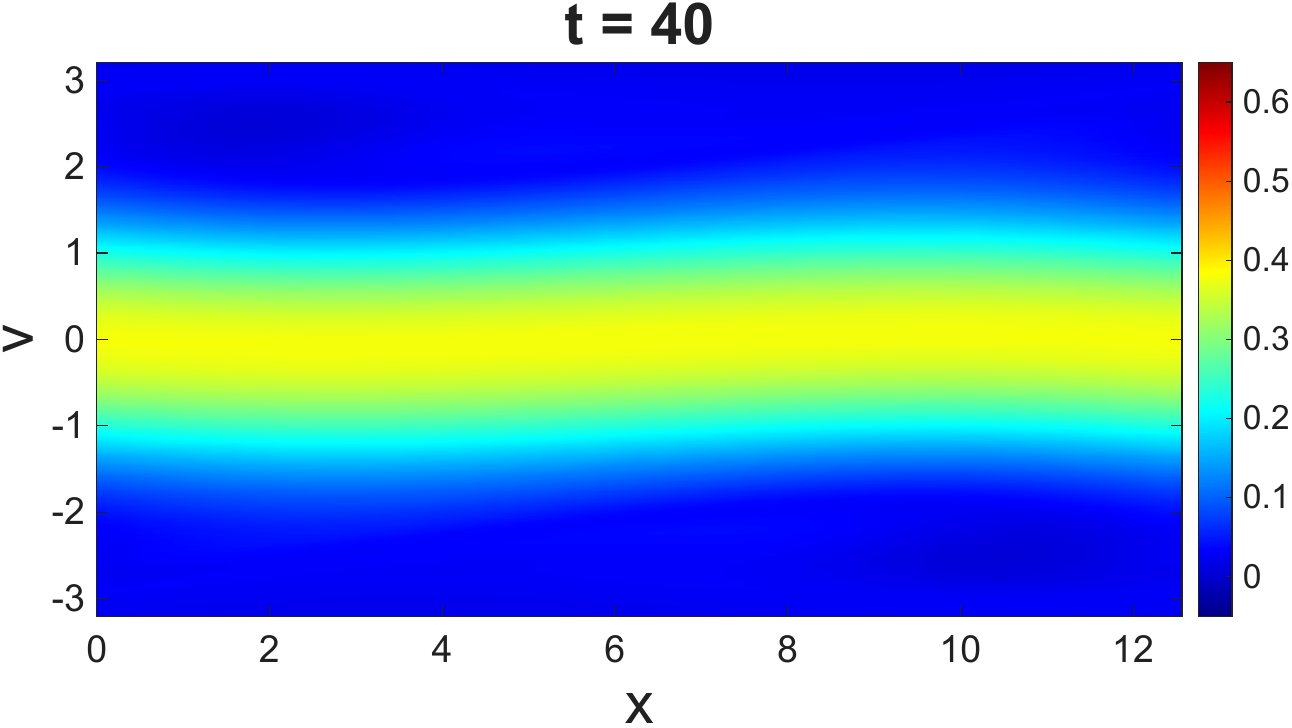}
        \subcaption{$\tau_0 = 10^3$.}
    \end{subfigure}
    \begin{subfigure}[t]{0.32\textwidth}
        \centering
        \includegraphics[width=\linewidth, height=4cm, keepaspectratio]{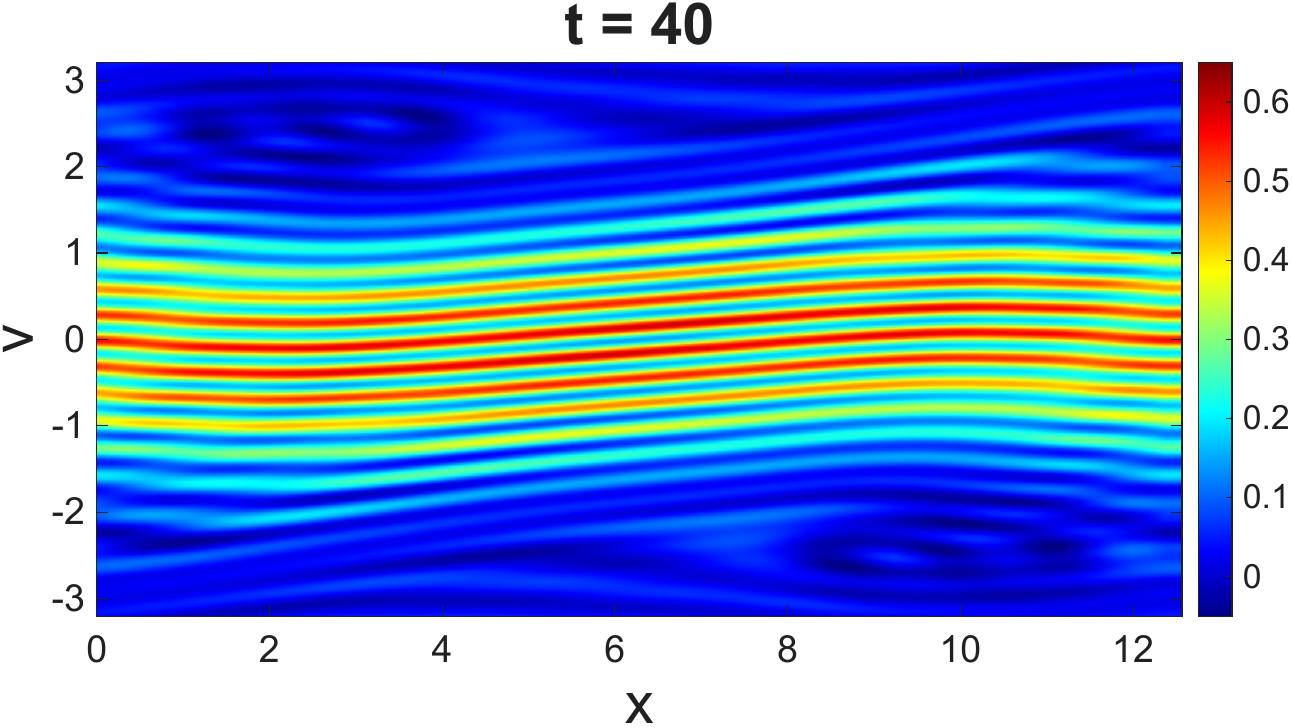}
        \subcaption{$\tau_0 = 10^5$.}
    \end{subfigure}

    \caption{Strong Landau damping: snapshots of the distribution function $f$ at $t \in \{4,16,40\}$ across the collisional regimes. $\tau_0\in\{10^1,10^3,10^5\}$, $(N_x, N_H) = (128, 640)$.}
    \label{fig:landau_compare_distribution}
\end{figure}

Figure~\ref{fig:landau_compare_energy} illustrates the time evolution of macroscopic quantities. Since the distance between the numerical solution and the equilibrium is measured by the functional $\mathcal{E}_h(t)$ defined in \eqref{Eh:0}, we monitor the electric energy $\xnorm{E}$ and the deviation $\norm{f-f_\infty}_{L^2(f_\infty^{-1})}$ of the distribution function $f$ from equilibrium $f_\infty$. Besides, we report the evolution of the distances to the macroscopic equilibrium $\rho_\infty$ and to the local equilibrium $\rho\mathcal{M}$, in order to further study the underlying relaxation mechanisms. In the weakly collisional regime, none of these quantities exhibits a clear decay trend, as transport and electric effects dominate the dynamics. As the collision strength increases, dissipation increases significantly, driving the solution toward equilibrium. Moreover, the macroscopic dissipation exhibits regular oscillations. These numerical observations are consistent with the analytical results established in Section~\ref{sec:main_results}, although the amplitude $\delta=0.5$ lies beyond the $\tau_0$-dependent smallness assumption \eqref{eq:vpfp1d_hermite_semidiscrete_hypocoercive_condition} under which exponential relaxation toward equilibrium is guaranteed. 

\begin{figure}[!htb]
    \centering
    \begin{subfigure}[t]{0.48\textwidth}
        \centering
        \includegraphics[width=\linewidth, height=4.8cm, keepaspectratio]{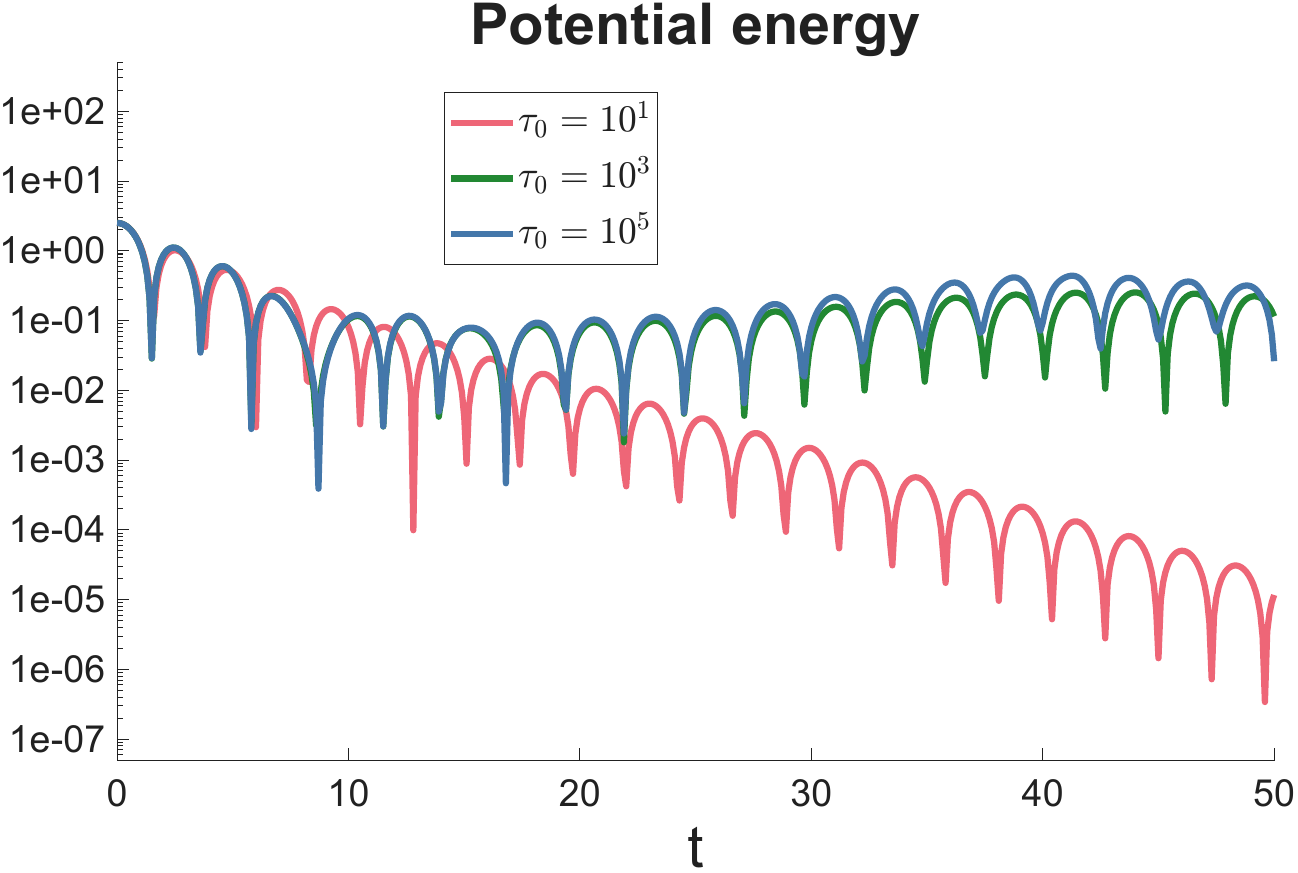}
        \subcaption{$\xnorm{E}$}
    \end{subfigure}\hfill
    \vspace{0.8em}
    \begin{subfigure}[t]{0.48\textwidth}
        \centering
        \includegraphics[width=\linewidth, height=4.8cm, keepaspectratio]{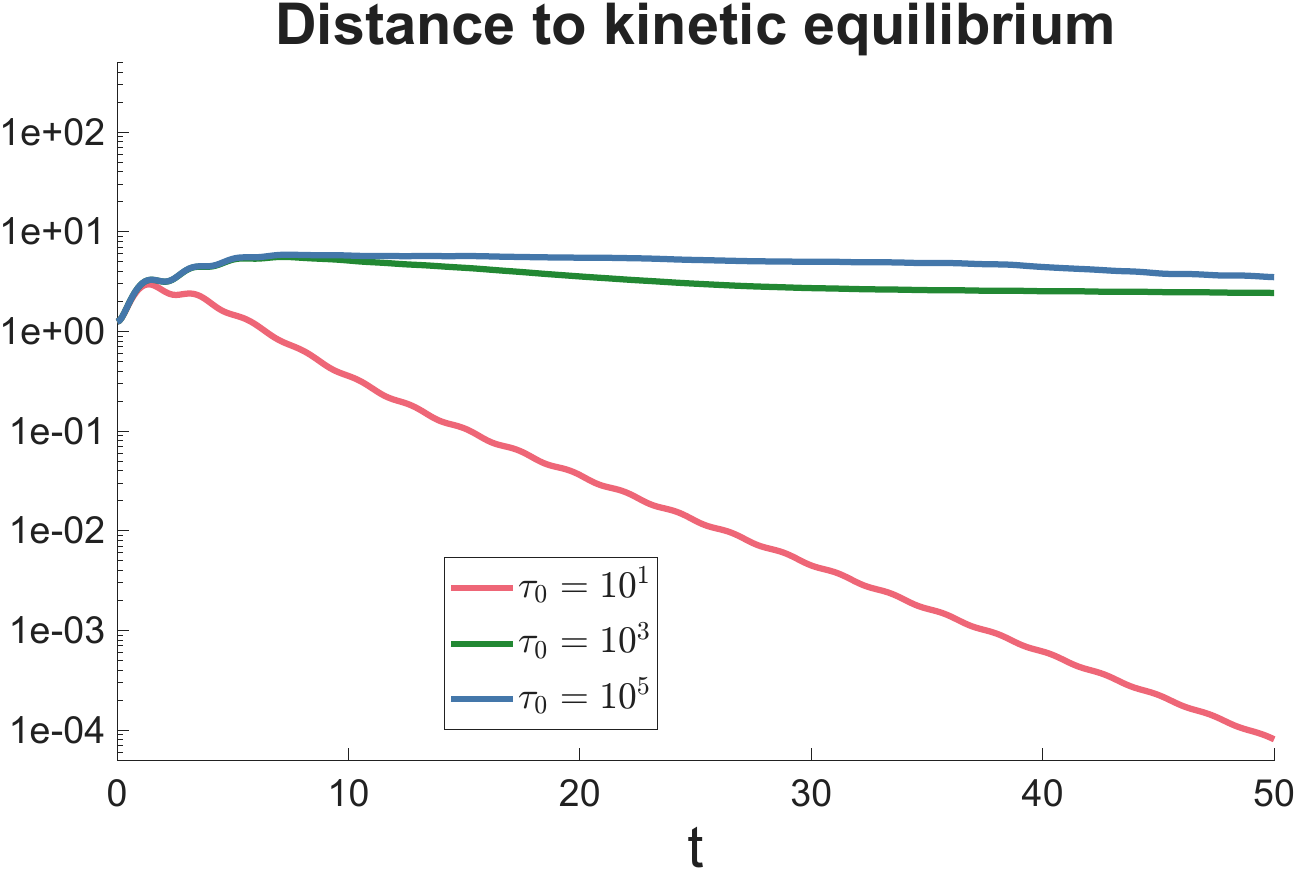}
        \subcaption{$\norm{f-f_\infty}_{L^2(f_\infty^{-1})}$}
    \end{subfigure}
    \vspace{0.8em}
    \begin{subfigure}[t]{0.48\textwidth}
        \centering
        \includegraphics[width=\linewidth, height=4.8cm, keepaspectratio]{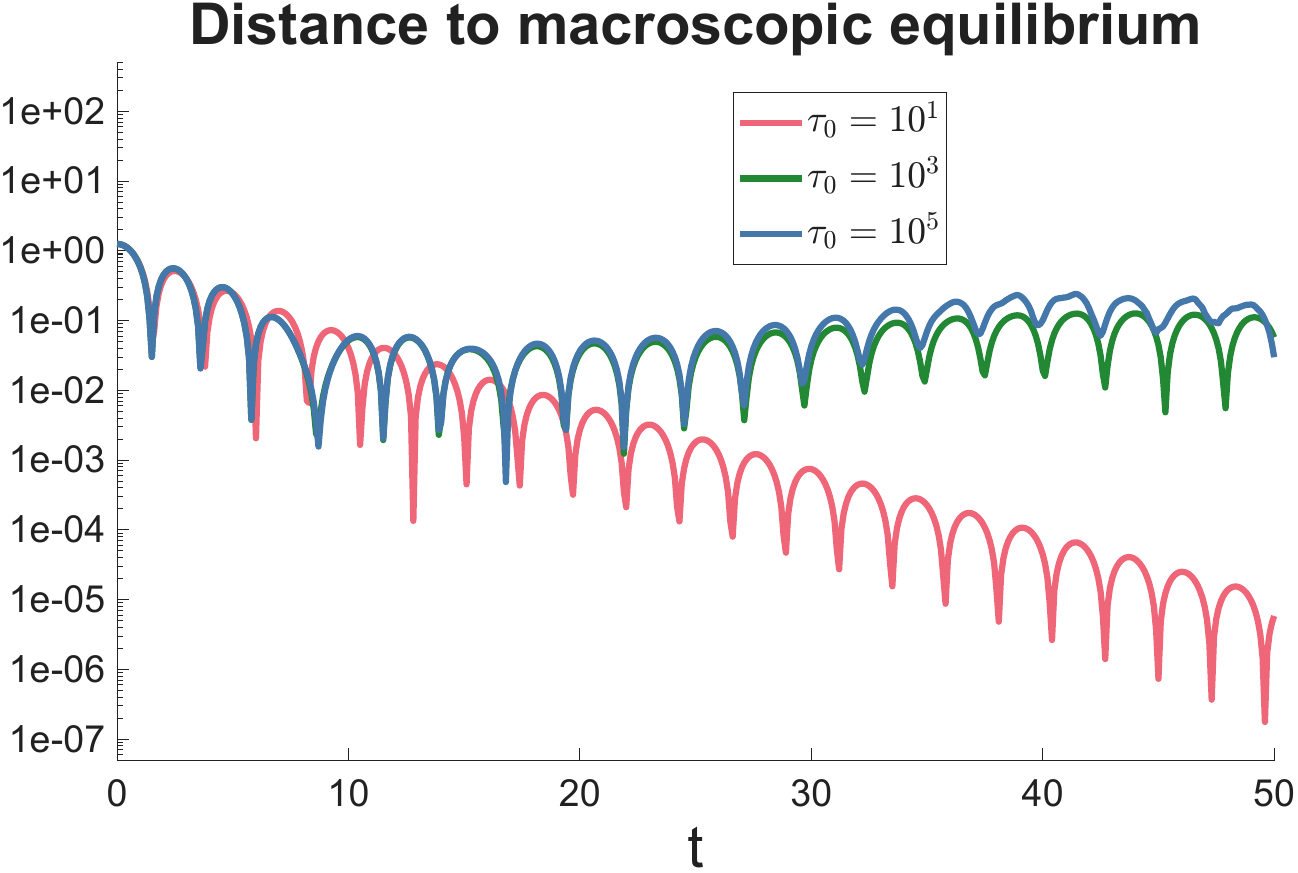}
        \subcaption{$\norm{\rho-\rho_\infty}_{L^2(\rho_\infty^{-1})}$}
    \end{subfigure}\hfill
    \begin{subfigure}[t]{0.48\textwidth}
        \centering
        \includegraphics[width=\linewidth, height=4.8cm, keepaspectratio]{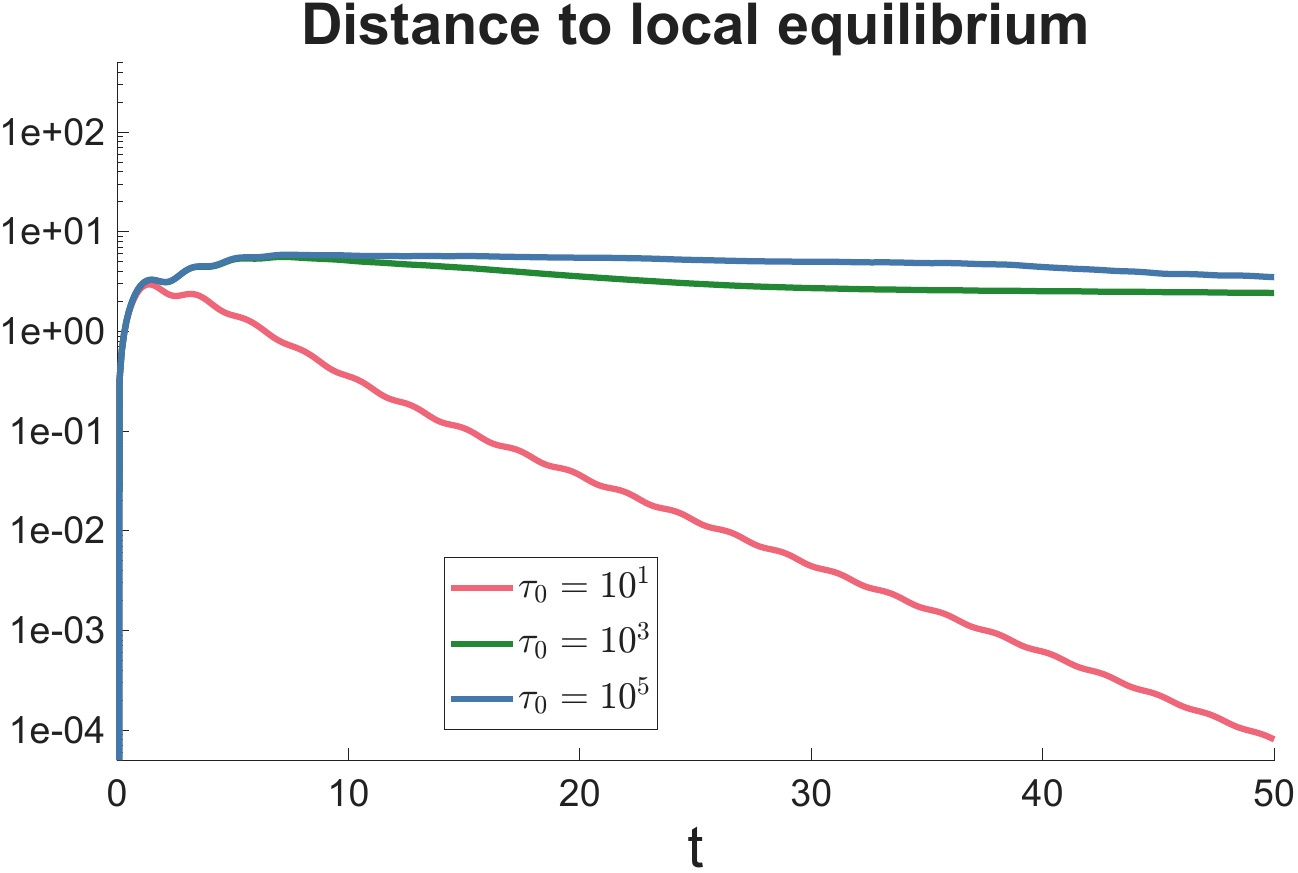}
        \subcaption{$\norm{f-\rho\mathcal{M}}_{L^2(f_\infty^{-1})}$}
    \end{subfigure}
    \caption{Strong Landau damping: time evolution in logarithmic scale of (a) the electric energy, (b) the distance to kinetic equilibrium, (c) the distance to macroscopic equilibrium, and (d) the distance to local equilibrium across the collisional regimes. $\tau_0\in\{10^1,10^3,10^5\}$, $(N_x,N_H) = (128,640)$.}
    \label{fig:landau_compare_energy}
\end{figure}

\section{Conclusion and perspectives}
\label{sec:conclusion}

In this work, we analyze a numerical scheme for the nonlinear Vlasov-Poisson-Fokker-Planck kinetic model. More precisely, we prove that our discrete solutions relax exponentially toward the thermal equilibrium under some smallness condition on the initial data. Our numerical method is based on a Hermite--discontinuous Galerkin discretization for the velocity and spatial variables respectively. The analysis relies on the structure preserving properties of the method, which allow us to establish hypocoercive estimates at the discrete level. To control the nonlinear contributions, we also prove discrete elliptic regularity along with $L^\infty$ estimates for the electric field.

Several perspectives arise from this work. A natural extension
concerns the study of fully discrete schemes and the preservation of
hypocoercive estimates under high-order time discretizations. Another
direction consists in extending the present framework to
higher-dimensional settings, where additional theoretical tools may be
needed to address nonlinear field
interactions. {Indeed, discrete hypocoercivity methods
  will be more intricate to implement since they may require
  propagating discrete regularity, as is the case in the continuous
  setting. {To do so, a key step will be to adapt Lemma
    \ref{lem:ops_ah_props} and the Sobolev embedding
    \eqref{eq:ops_bh_electric_regularity_estimate} in the higher
    dimensional framework. Such results have already been applied to
    prove error estimates for numerical methods for the Vlasov-Poisson
    model in \cite[Sections $A.3$ and $A.1$]{deDios_Carrillo_Shu12}.}
  Furthermore, time discretization is central to avoid prohibitive
  computational costs in the multi-dimensional framework. A first step
  was made in \cite{blaustein_structure_2024}, where a provably
  structure preserving and efficient time splitting strategy was
  proposed. To the best of our knowledge, taking this additional layer
  of complexity into account in the analysis has not yet been done.}
It would also be of interest to investigate discontinuous Galerkin
schemes for nonlinear kinetic models in the diffusive regime{, since the current literature is restricted to the linear setting (see for instance \cite{Lemou_Mieussens09,bessemoulin-chatard_hypocoercivity_2020,blaustein_discrete_2024}).} Finally, the theoretical tools developed here may be adapted to study the long-time behavior of other kinetic models, such as those involving scattering or nonlinear Fokker-Planck collisions.



\section*{Acknowledgments}

Yi Cai gratefully acknowledges the financial support of the Graduate School of Xiamen University for his visit to Toulouse. Yi Cai and Tao Xiong are supported by the National Key R\&D Program of China No. 2022YFA1004500 and by the NSFC grant No.12571443. Alain Blaustein acknowledges the support of the CDP C2EMPI, together with the French State under the France-2030 programme, the University of Lille, the Initiative of Excellence of the University of Lille, the European Metropolis of Lille for their funding and support of the R-CDP-24-004-C2EMPI project. Alain Blaustein and Francis Filbet are supported by the ANR-25-CE40-5565 project Cookie.

\section*{Appendix A. Discrete properties }
\label{App:disc:inv}

\subsection*{Proof of Proposition \ref{prop1:disc:inv}}
To prove the conservation of mass, we take the $L^2$ product between the first line in \eqref{eq:vpfp1d_hermite_semidiscrete_riesz} and the test function $u_0\,=\,1\in U^m_h$. Using the duality relation \eqref{eq:ops_ah_discrete_duality} and the kernel property \eqref{eq:ops_ah_discrete_kernel}, we obtain the result
\[
\frac{\dd{m_{h,0}}}{\dd{t}}\,=\,
\frac{\dd{}}{\dd{t}}
\xhdual{D_{h,0}}{1}\,=\, 0\,.
\]
To prove item \ref{disc:momentum}, we take the $L^2$ product between the second line in \eqref{eq:vpfp1d_hermite_semidiscrete_riesz} and the test function $u_1\,=\,1\in U^m_h$. Using the duality relation \eqref{eq:ops_ah_discrete_duality}, the kernel property \eqref{eq:ops_ah_discrete_kernel} and since it also holds $\mathcal{A}_h^* u_1\,=\,0$, we obtain
\[
\frac{\dd{m_{h,1}}}{\dd{t}}\,=\,
\sqrt{T_0}\,\frac{\dd{}}{\dd{t}}\xhdual{D_{h,1}}{1}\,=\, 
\xhdual{\Pi_h\left( E_h(D_{h,0}-D_{\infty,0})\right)}{1} \,-\, \frac{1}{\tau_0} m_{h,1}\,.
\]
Since $u_1\,=\,1\in U^m_h$, we remove $\Pi_h$ in the previous right hand side and then we substitute $D_{h,0}-D_{\infty,0}$ according to the second equation in the fourth line of \eqref{eq:vpfp1d_hermite_semidiscrete_riesz}
\begin{equation}\label{App:estim:momentum}
\frac{\dd{m_{h,1}}}{\dd{t}}\,=\, 
-\,
\xhdual{E_h}{\mathcal{B}_h^* E_h} \,-\, \frac{1}{\tau_0} m_{h,1}
\,.
\end{equation}
When $\mathcal{B}_h^*$ corresponds to a discontinuous Galerkin approximation of $-\partial_x$, we rewrite the right hand side according to \eqref{eq:ops_bh_ldg_global} with $(E,v)\,=\,(E_h,E_h)$, which is allowed since $E_h\in U_h^m$ has zero mean, hence $E_h\in V_h$, leading to
\[
\xhdual{E_h}{\mathcal{B}_h^* E_h}
\,=\, 
\sum_{j\in\mcJ}\hat{g}_{b,j-\half}^*(E_h)\jmp{E_h}_{j-\half} + \sum_{j\in\mcJ}\xKdual[K_j]{E_h}{\partial_x E_h}\,.
\]
We integrate by parts in the second sum on the right hand side
and re-index the subsequent boundary terms which yields
\begin{align*}
\xhdual{E_h}{\mathcal{B}_h^* E_h}
\,=\, 
&
\sum_{j\in\mcJ}\left[
\hat{g}_{b,j-\half}^*(E_h)
\jmp{E_h}_{j-\half}
+ \frac{1}{2}\left( 
|E_{j-\half}^-|^2
-
|E_{j-\half}^+|^2\right)\right]\,.
\end{align*}
Then, we rewrite the last term according to the identity 
$$
|E_{j-\half}^-|^2
-
|E_{j-\half}^+|^2
=
-
(
\hat{g}_{b,j-\half}(E_h)
+
\hat{g}_{b,j-\half}^*(E_h))
\jmp{E_h}_{j-\half},
$$
and we get that
\begin{align*}
\xhdual{E_h}{\mathcal{B}_h^* E_h}\,=\, 
&
\frac{1}{2}\sum_{j\in\mcJ}(
\hat{g}_{b,j-\half}^*
-
\hat{g}_{b,j-\half})(E_h)
\jmp{E_h}_{j-\half}\,.
\end{align*}
With the choice $\hat g_b(\Phi)=\Phi^-$ in \eqref{eq:ops_bh_ldg_flux}, that is $\hat g_b^*(E)=E^+$, we obtain
\begin{equation}\label{identity:Eh:minus}
\xhdual{E_h}{\mathcal{B}_h^* E_h}\,=\, \frac{1}{2}\;\sum_{j\in\mcJ}
\jmp{E_h}_{j-\half}^2 \geq 0,
\end{equation}
whereas, with the choice $\hat g_b(\Phi)=\Phi^+$, that is $\hat g_b^*(E)=E^-$, it holds
\begin{equation}\label{identity:Eh}
	\xhdual{E_h}{\mathcal{B}_h^* E_h}\,=\, -\frac{1}{2}\;\sum_{j\in\mcJ}
	\jmp{E_h}_{j-\half}^2 \leq 0.
\end{equation}
We deduce item \ref{disc:momentum} in Proposition \ref{prop1:disc:inv} by substituting the previous relations in \eqref{App:estim:momentum}.

To prove item \ref{disc:energy}, we take the $L^2$ product between the third line in \eqref{eq:vpfp1d_hermite_semidiscrete_riesz} and the test function $u_2\,=\,1\in U^m_h$, we obtain
\begin{equation}\label{ener:proof:interm}
\frac{\dd{}}{\dd{t}}
\xhdual{D_{h,2}}{1}\,=\, 
\sqrt{\frac{2}{T_0}}
\xhdual{E_h}{D_{h,1}} - \frac{2}{\tau_0} \xhdual{D_{h,2}}{1}
\,.
\end{equation}
Then, we substitute $E_h$ according to the first equation in the fourth line in \eqref{eq:vpfp1d_hermite_semidiscrete_riesz} and use that, in this case $(\hat{g}_a,\hat{g}_a^*)=(\hat{g}_b,\hat{g}_b^*)$, it holds $\sqrt{T_0}\,\mathcal{B}_h\,=\, \mathcal{A}_h$ on $V_h$ to get
\[
\sqrt{\frac{2}{T_0}}
\xhdual{E_h}{D_{h,1}} 
\,=\, 
-\frac{\sqrt{2}}{T_0}\,
\xhdual{\mathcal{A}_h \Phi_h}{D_{h,1}}
\,.
\]
Then, we reformulate the right hand side following the same lines as in the proof of Proposition \ref{prop:4:2} to compute $\mathcal{D}_h$, which yields
\[
 \sqrt{\frac{2}{T_0}}
\xhdual{E_h}{D_{h,1}} 
\,=\, 
-\frac{1}{\sqrt{2}\,T_0}\,
\frac{\dd{}}{\dd{t}}\|E_h\|_{L^2(\mathbb{T})}^2
\,.
\]
Using the definition of the kinetic energy $K_{h}$, the conservation of mass and \eqref{ener:proof:interm}, we get the expected result, since $2K_h-T_0\,m_{h,0}=\sqrt{2}\,T_0\xhdual{D_{h,2}}{1}$.

\subsection*{Proof of Proposition \ref{prop2:disc:inv}}
We have already proven the conservation of mass in the previous proof. For the total momentum estimate, the same arguments as in the previous proof lead to \eqref{App:estim:momentum}, that is
\[
\frac{\dd{m_{h,1}}}{\dd{t}}\,=\, 
-\,
\xhdual{E_h}{\mathcal{B}_h^* E_h} \,-\, \frac{1}{\tau_0} m_{h,1}
\,.
\]
In the Raviart-Thomas case \eqref{eq:ops_bh_rt}, $E_h\in W_h= U_h^{m+1} \cap C^0(\mbT)$ has a weak derivative in $L^2$ which satisfies $\partial_x E_h \in V_h$ (the zero mean condition follows from the periodicity of $E_h$). Furthermore, according to \eqref{eq:ops_bh_rt} it holds for all $v\in V_h$
\[
-\xhdual{\partial_x E_h}{v}\,=\,
\xhdual{\mathcal{B}_h^* E_h}{v}\,,
\]
which justifies that $\mathcal{B}_h^* E_h\,=\, -\partial_x
E_h$. Therefore, it holds $\xhdual{E_h}{\mathcal{B}_h^* E_h}= -\xhdual{E_h}{\partial_x E_h}=0$ and we deduce
\[
\frac{\dd{m_{h,1}}}{\dd{t}}\,=\, \,-\, \frac{1}{\tau_0} m_{h,1}\,.
\]

\bibliographystyle{amsplain}
\bibliography{references.bib}

\end{document}